\documentclass[11pt]{article}
\usepackage{epsfig, amsmath, amssymb, amsthm, times,color}

\usepackage[utf8]{inputenc}

\usepackage{amsmath,amssymb,amsmath}
\usepackage{mathabx}
\usepackage{graphicx}
\usepackage{subfigure}
\usepackage{cite}
\usepackage{hyperref}
\usepackage{url}
\usepackage{mathrsfs}

\numberwithin{equation}{section}

\newcommand{\pv}{\operatorname{p.v.}}

\newcommand{\R}{\mathbb{R}}
\newcommand{\T}{\mathbb{T}}
\newcommand{\C}{\mathbb{C}}
\newcommand{\Z}{\mathbb{Z}}
\newcommand{\N}{\mathbb{N}}

\newcommand{\bW}{\boldsymbol{W}}
\newcommand{\bL}{\boldsymbol{L}}
\newcommand{\bB}{\boldsymbol{B}}
\newcommand{\bI}{\boldsymbol{I}}
\newcommand{\bS}{\boldsymbol{S}}
\newcommand{\bP}{\boldsymbol{P}}
\newcommand{\bR}{\boldsymbol{R}}

\newcommand{\cH}{\mathcal{H}}
\newcommand{\cF}{\mathcal{F}}
\newcommand{\cD}{\mathcal{D}}

\newcommand{\cP}{\mathcal{P}}
\newcommand{\cR}{\mathcal{R}}

\newcommand{\bbH}{\mathbb{H}}

\def\lbl{\label}
\def\be{\begin{equation}}
\def\ee{\end{equation}}

\newcommand{\p}{{\partial}}
\def\iu{{\mathrm{i}\mkern1mu}}

\def\1{\mathbf{1}}

\newtheorem{theorem}{Theorem}[section]

\newtheorem{lemma}[theorem]{Lemma}

\newtheorem{remark}[theorem]{Remark}
\newtheorem{definition}[theorem]{Definition}
\newtheorem{assumption}[theorem]{Assumption}

\title{Stability and bifurcation of mixing in the Kuramoto model with inertia}
\date{\today}
\author{Hayato Chiba\thanks{Advanced Institute for Materials Research,
    Tohoku  University, Sendai, 980-8557, Japan, {\tt hchiba@tohoku.ac.jp}}\;
  \quad and\quad
Georgi S. Medvedev\thanks{Department of Mathematics, 
Drexel University, 3141 Chestnut Street, Philadelphia, PA 19104,
{\tt medvedev@drexel.edu}} 
}

\begin{document}
\maketitle
\begin{abstract}
  The Kuramoto model of coupled second order damped oscillators 
  on convergent sequences of graphs is analyzed in this work.
  The oscillators in this model have random intrinsic frequencies and
  interact with each other via  nonlinear coupling. The connectivity
  of the coupled system is assigned by a graph which may be random as
  well.
  In the thermodynamic limit the behavior
  of the system is captured by the Vlasov equation, a hyperbolic
  partial differential equation for the
  probability distribution of the oscillators in the phase space.
  We study stability of mixing, a steady state solution of the Vlasov equation, corresponding
  to the uniform distribution of phases. Specifically, we identify a critical value of the
  strength of coupling, at which the system undergoes a pitchfork
  bifurcation. It corresponds to the loss of stability of mixing and marks the
  onset of synchronization. As for the classical Kuramoto model,
  the presence of the continuous spectrum on the imaginary axis poses the main difficulty for the
  stability analysis. To overcome this problem, we use the methods from the generalized spectral
  theory developed for the original Kuramoto model. The analytical results are
  illustrated with numerical bifurcation diagrams computed for the Kuramoto model
  on Erd{\H o}s--R{\' e}nyi and small-world graphs.
  
  Applications of the second-order Kuramoto model include power networks, coupled pendula, and
  various biological networks. The analysis in this
  paper provides a mathematical description of the onset of synchronization in these systems.
  \end{abstract}
  
  %%%%%%%%%%%%%%%%%%%%%%%%%%%%%%%%%%%%%%%%%%%%%%%%%%%%%%%%%%%%%%%
  
\section{Introduction}
\setcounter{equation}{0}

In this paper, we study the following system of coupled second order damped oscillators on a
convergent sequence of
graphs $\{\Gamma_n\}$:
\begin{equation}\label{2nd-order}
  \ddot\theta _i+ \gamma  \dot\theta _i= \omega _i +
  \frac{2K}{n} \sum^n_{j=1}a^n_{ij} \sin \left(\theta _j-\theta _i\right),
\quad i\in [n]=\{1,2,\dots,n\},
\end{equation}
where $\theta _i, i\in [n],$ denotes the phase of the $i$th oscillator,
$\gamma >0$ is a damping constant, $K$ is the coupling strength and
$a^n_{ij}$ is the adjacency matrix of $\Gamma_n$. Intrinsic frequencies $\omega_i, i\in [n],$
are independent identically distributed  random variables drawn from the probability
       distribution with density $g$.
By rescaling time, intrinsic frequencies, and $K$, one can make $\gamma
=1$, which will be assumed without loss of generality throughout this paper.

System of ordinary differential equations \eqref{2nd-order} may be viewed as an extension of the
classical Kuramoto model (KM), which is used to study collective behavior in large coupled systems \cite{Kur84}.
In this context, it is called the KM with inertia, because the second order terms in \eqref{2nd-order}
introduce inertial effects into the system's dynamics.
On the other hand, system \eqref{2nd-order} has a clear mechanical interpretation. For instance,
it may be viewed as a system of coupled pendula or coupled power generators, as
it is perhaps best known for
\cite{SMV84,TOS19}. Like its classical counterpart, the KM with inertia
features transition to synchronization. For small positive values of the coupling strength, the 
system exhibits irregular  dynamics, called mixing. Upon increasing the coupling strength
$K$, mixing loses stability and a gradual (soft) transition to synchronization takes place.
The onset of synchronization in the classical KM with all-to--all coupling was analyzed in \cite{Chi15a, Die16}
and in the model on graphs in \cite{ChiMed19a, ChiMed19b}.
The analysis  of synchronization in the KM with inertia  involves new technical challenges 
related to the dimensionality of the phase space, which make it an interesting stability
problem. In addition, synchronization in the KM with inertia has important applications.
First, 
synchronization
provides a more efficient regime of operation of power networks. Therefore, it is important to know
what factors determine stability of synchrony in the second order model on graphs.
Previous studies of the onset of synchronization in the  KM with inertia \cite{TLO97, TSL97}
relied on asymptotic analysis and numerical simulations. In this work, for the first time we perform
a rigorous linear stability analysis of mixing and identify the instability of mixing.
Furthermore, we perform a formal center manifold
reduction
and identify the bifurcation, marking the onset
of synchronization, as a pitchfork bifurcation. We verified these results  numerically for the KM on
various graphs including Erd\H{o}s-R{\' e}nyi (ER) and small-world (SW) graphs.

To proceed, we rewrite \eqref{2nd-order} as a system of coupled second order ordinary
differential equations and add a separate equation for $\omega_i$:
\be\lbl{rewrite-2nd}
\begin{split}
\dot\theta _i &  = \psi_i + \omega _i, \\
\dot\psi_i  &= -\psi_i + \frac{2K}{n} \sum^n_{j=1}a^n_{ij} \sin (\theta _j-\theta _i),\\
\dot\omega_i& =0,\qquad i\in [n].
\end{split}
\ee

To study  the dynamics of \eqref{rewrite-2nd} for large $n$, we employ the following Vlasov equation
\begin{equation}\label{Vlasov}
\partial_t f + \partial_\theta \left( (\psi + \omega ) f \right)
+\partial_\psi \left( (-\psi + \mathbf{N}[f])f \right) = 0,
\quad f=f(t,\theta,\psi,\omega,x),
\end{equation}
where
\begin{equation*}
\displaystyle \mathbf{N}[f](t, \theta ,x) = \frac{K}{\iu}\left( e^{-\iu\theta } h(t,x) - e^{\iu\theta } \overline{h(t,x)} \right). 
\end{equation*}
Here and below, we use $\iu =\sqrt{-1}$ to denote the imaginary unit.  
% We will continue to use 
% $i$ to
% denote integer valued indices. This will not lead to any confusion, as the meaning of $i$
% should be clear from the context. 
Further,
$f(t,\theta,\psi,\omega,x) d\theta d\psi d\omega$ is the probability that 
the state of the oscillator at point $x\in [0,1]$ at time $t\in \R^+$ is
in $(\theta, \theta+d\theta)\times (\psi, \psi+d\psi)\times (\omega, \omega+ d\omega)$.
The justification of the Vlasov equation as a mean field limit  for the original KM
on graphs can be
found in \cite{KVMed18, ChiMed19a}. It extends verbatim to cover the model at hand.
  
The following local order parameter plays a key role in the analysis of synchronization in the KM model:
\begin{equation}\lbl{def-order}
h(t,x) = \int_{\T\times\R^2 \times I} W(x,y) e^{\iu \theta } f(t,\theta ,\psi, \omega ,y) d\theta d\psi d\omega dy,
\end{equation}
where $I=[0,1]$ and $\T=\R/2\pi\Z$. $W$ is a square integrable function on $I^2$, which describes the limit of the
graph sequence $\{\Gamma_n\}$. $W$ is called a graphon. For the details on graph limits and their
applications to the continuum description of the dynamical networks, we refer the interested
reader to \cite{ChiMed19a, Med19}.

A weak (distribution--valued) solution $f(t,\theta,\psi,\omega,x)$ of the initial value problem for the Vlasov equation \eqref{Vlasov}
yields the probability distribution of particles in the phase space $\T\times \R^2$ for each
$(t,x)\in \R^+\times I$ (cf. \cite{Dob79, KVMed18}).  By mixing we call the following
steady state solution of \eqref{Vlasov}
\be\lbl{mixing}
f_{mix}=\frac{g(\omega)}{2\pi}\delta(\psi),
\ee
where $\delta$ stands for Dirac's delta function
and $g(\omega)$ is the probability density function characterizing the distribution
of intrinsic frequencies. $f_{mix}$ corresponds to the stationary regime
when the values of $\theta_i$ are distributed uniformly over $\T$. Stability of mixing is the main theme
of this paper.

In the next section, we recast the stability problem in the Fourier space and derive the linearized problem.
The next step is to characterize the spectrum of the linearized operator. This is done in
Section~\ref{sec.spectral}. We start by describing the setting for the spectral problem. To this end,
we define function spaces and operators involved in the linear stability problem.  After that we derive
a transcendental equation for the eigenvalues of the linearized problem.  The analysis of the
eigenvalue problem shows that the  bifurcating eigenvalue fails to cross the imaginary
axis filled by the residual spectrum. To overcome this difficulty, following the analysis of the
classical KM \cite{Chi15a, ChiMed19b},
we use the weak formulation for the eigenvalue problem for the linearized operator
based on a carefully selected Gelfand triplet and the analytic continuation of the
resolvent operator past the imaginary axis. In this setting the generalized
eigenvalues are defined as singular points of the generalized
resolvent. The corresponding eigenfunction is used in the center manifold reduction of the
system's dynamics. This concludes a rigorous linear stability analysis
  of mixing and sets the scene for the bifurcation analysis.
  In Section~\ref{sec.stability}, we prove that mixing is asymptotically
  stable in the subcritical regime.  In Section~\ref{sec.bifurcation}, we show that mixing undergoes a
  pitchfork bifurcation marking the onset of synchronization. The center manifold reduction
  is done under the assumption of the existence of the center manifold.  However, the formal center manifold
  reduction alone is a formidable problem, as the analysis in this section shows.
  The proof of existence of the center manifold requires new additional techniques and 
  is beyond the scope of this paper. The analytical results are
  illustrated with numerical examples of the bifurcation in the KM on ER and SW 
  graphs in Section~\ref{sec.numerics}.

%%%%%%%%%%%%%%%%%%%%%%%%%%%%%%%%%%%%%%%%%%%%%%%%%%%%%%%%%%%%%

\section{The Fourier transform}\lbl{sec.Fourier}
\setcounter{equation}{0}
As a probability density function, $g$ is a nonnegative integrable function
such that
\be\lbl{int-g}
\int_\R g(s)ds =1.
\ee
Since $\omega$ does not change in time, $f$ satisfies the following constraint
\begin{eqnarray}\label{g-constraint}
g(\omega ) = \int_{\T \times \R}\!  f(t,\theta ,\psi, \omega ,x)
  d\theta d\psi, \quad \forall (t,x)\in \R^+\times I.
\end{eqnarray}
In addition, throughout this paper, we will assume the following.
\begin{assumption}\label{as.decay} Let $g$ be a real analytic function. In addition, let
  the Fourier transform $\hat{g}(\eta) :=\int_\R e^{\iu \eta\omega}g(\omega) d\omega$ be continuous on $\R$ and
  \be\lbl{g-decay}
  \lim_{\eta\to\infty} |\hat{g}(\eta)|e^{a\eta}=0
  \ee
  for some $0<a<1$.
  \end{assumption}
  \begin{remark}
By the Paley-Wiener theorem, \eqref{g-decay} implies that    
$g(\omega )$ has an analytic continuation to the region $0<\mathrm{Im}(z) < a$. 
\end{remark}

In preparation for the stability analysis of mixing, we rewrite \eqref{Vlasov} using Fourier variables:
\begin{equation}\label{transform}
u_j(t, \zeta, \eta , x) = \int_{\R^2\times \T} e^{\iu (j\theta + \zeta \psi + \eta \omega  )} 
   f (t,\theta ,\psi, \omega ,x) d\theta d\psi d\omega,\quad j\in\Z.
 \end{equation}

 The Vlasov PDE is then rewritten as a system of differential equations for the Fourier coefficients
 $u_j, \, j\in\Z$:
\begin{equation}\label{Fourier}
  \begin{split}
\p_tu_j &= 
-\int e^{\iu (j\theta + \zeta \psi + \eta \omega  )}\frac{\partial }{\partial \theta } 
\left\{ (\psi + \omega ) f \right\} d\theta d\psi d\omega  \\
&  - \int e^{\iu (j\theta + \zeta \psi + \eta \omega  ) }
          \frac{\partial }{\partial \psi}
          \left\{(-\psi + \mathbf{N}[f]) f \right\} d\theta d\psi d\omega  \\
&= \iu j \int e^{\iu (j\theta + \zeta \psi + \eta \omega  ) } (\psi + \omega ) f d\theta d\psi d\omega \\
&  \quad +\iu \zeta \int e^{\iu (j\theta + \zeta \psi + \eta \omega  ) }
         \left( -\psi + \frac{K}{\iu} (e^{-\iu \theta } h - e^{\iu \theta } \overline{h}) \right) fd\theta d\psi d\omega \\
&= (j-\zeta) \frac{\partial u_j}{\partial \zeta} + j \frac{\partial u_j}{\partial \eta}
     + K\zeta (h(t,x) u_{j-1} - \overline{h(t,x)} u_{j+1}),\; j\in\Z,
\end{split}
   \end{equation}
where
\begin{equation}
h(t,x) = \int_I W(x,y) u_1(t,0,0 ,y) dy.
\end{equation}
is the local order parameter.
In addition, we have the following constraints
\begin{eqnarray}\label{F-const-a}
  u_0(t,0,0 ,x) &=& 1,\\
  \label{F-const-b}
  u_{-j}(t,-\zeta, -\eta ,x) & =& \overline{u_j (t,\zeta,\omega ,x)}, \; j\in \N.
\end{eqnarray}
Here,  \eqref{F-const-a} follows from \eqref{g-constraint} and 
\eqref{F-const-b} follows from the fact that $f$ is real. Thus,
it is sufficient to restrict to $j\in \N\bigcup\{0\}$ in \eqref{Fourier}.

By changing from $\zeta$ to $\xi$ given by the following relations:
\begin{equation}
\left\{ \begin{array}{ll}
\zeta - j = - e^{-\xi_j}, & \zeta -j<0, \\
\zeta - j =   e^{-\xi_j}, & \zeta -j\geq 0, \\
\end{array} \right.
\label{1-6}
\end{equation}
and setting
\begin{equation}
v_j (t, \xi_j, \eta ,x) := 
\left\{ \begin{array}{ll}
u_j(t, j- e^{-\xi_j }, \eta , x), & \zeta -j<0, \\[0.2cm]
u_j(t, j+ e^{-\xi_j }, \eta , x), & \zeta -j\geq 0, \\
\end{array} \right.
\end{equation}
we obtain
\begin{equation*}
  \begin{split}
\partial_t v_j
  &=\partial_{\xi_j} v_j 
    +j \partial_\eta v_j+K(j-e^{-\xi_j}) \left( h(t,x)v_{j-1}-\overline{h(t,x)} v_{j+1}\right), \\
\displaystyle h(t,x)& = \int_I  W(x,y) v_1(t,0,0,y)dy, 
                         \end{split}                       
\end{equation*}
subject to the constraint $v_0(t, \infty , \eta ,x) = \hat{g}(\eta)$.
For $j\geq 0$, we adopt the first line of (\ref{1-6}) because in the definition of the local order parameter, we need $u_1(t,0,0,x)$, for which $\zeta -j = -1<0$.
By the same reason, we use the second line for $j \leq -1$.

The steady state of the Vlasov equation, $f_{mix}$,  in the Fourier space has the following form
\be\lbl{Fmix}
v_0 = \hat{g}(\eta), \quad v_j = 0, j\in\N.
\ee

To investigate the stability of \eqref{Fmix}, let $w_0 = v_0 - \hat{g}(\eta)$ and $w_j = v_j$ for $j\neq 0$.
Then we obtain the system
\begin{equation}
\left\{ \begin{array}{l}
\displaystyle \partial_t w_1= \partial_{\xi_1} w_1  + \partial_\eta w_1
     + K(1- e^{-\xi_1}) \left( h(t,x) \hat{g}(\eta) + h(t,x)w_0 - \overline{h(t,x)} w_2 \right), \\[0.4cm]
\displaystyle \frac{dw_j}{dt} = \frac{\partial w_j }{\partial \xi_j} + j\frac{\partial w_j}{\partial \eta}
          + K(j- e^{-\xi_j}) \left( h(t,x)w_{j-1} - \overline{h(t,x)} w_{j+1} \right), 
          \; j\ge 0\;\mbox{and}\; j\neq 1.\\
          [0.4cm]
\displaystyle h(t,x) = \int_I W(x,y) w_1(t,0 ,0 ,y) dy,
\end{array} \right.
\label{1-8}
\end{equation}
and $w_0(t, \infty , \eta ,x) = 0$. 
Our goal is to investigate the stability and bifurcations of the steady state (mixing)
$w_j = 0, \, j\in \Z$ of this system.

The linearized system has the following form
\begin{eqnarray}\label{linear}
\frac{dw_1}{dt} & = & \bL_1[w_1] + K \bB[w_1] =: \bS[w_1]  \\
\frac{dw_j}{dt} &=& \bL_j [w_j],\quad j\ge 0\;\mbox{and}\; j\neq 1, \label{linear-2}
\end{eqnarray}
where
\begin{eqnarray} \lbl{Lj}
 \bL_j[\phi](\xi, \eta , x) 
  &=& \left( {\partial_\xi} + j {\partial_\eta}\right) \phi(\xi, \eta, x), \; j\in \Z\\
  \label{def-B}
\bB[\phi](\xi, \eta , x)&=&(1-e^{-\xi}) \hat{g}(\eta) \bW[\phi (0,0, \cdot)](x),
\end{eqnarray}
and
\begin{equation}\lbl{def-W}
\bW[f](x) = \int_{\R} W(x,y) f(y) dy.
\end{equation}

%%%%%%%%%%%%%%%%%%%%%%%%%%%%%%%%%%%%%%%%%%%%%%%%%%%%%%%%%%%%%%%%%%%%%%%%%

\section{The spectral analysis}\lbl{sec.spectral}
\setcounter{equation}{0}
\subsection{The setting}
% Note that as a probability density function, $g$ is a nonnegative
% integrable function such that
% \be\lbl{int-g}
% \int_\R g(\omega)d\omega=1.
% \ee

To define the operators used in the linearized system \eqref{linear} and \eqref{linear-2}
formally, we introduce the following Banach spaces.
For $\alpha\in [0, 1)$, let
\begin{equation}\label{weights}
  \beta^+_1(\eta) = \max\{ 1, e^{\alpha\eta}\}, \quad \beta^-_1(\eta) = \min\{ 1, e^{\alpha\eta}\},
  \quad\mbox{and}\quad \beta_2(\xi) = \min\{ e^{\xi}, 1 \},
  \end{equation}
and define
\begin{equation}\label{spaces}
\begin{split}
  \mathcal{X}^\pm_\alpha  & = \{ \phi : \text{continuous on $\R$},\, 
       \| \phi \|_{\mathcal{X}^\pm_\alpha} := \sup_{\eta} \beta^\pm_1(\eta) |\phi (\eta)| < \infty \}, \\
\mathcal{Y}^\pm_\alpha  &= \{ \phi : \text{continuous on $\R^2$},\, 
       \| \phi \|_{\mathcal{Y}^\pm_\alpha} := \sup_{\xi, \eta} \beta^\pm_1(\eta)\beta_2(\xi) |\phi (\xi, \eta)| < \infty \}, \\
       \mathcal{H}^\pm_\alpha  &= L^2 (I; \mathcal{Y}^\pm_\alpha).
       \end{split}
\end{equation}
The norms on $\mathcal{H}^{\pm}_\alpha$ are defined by
\begin{equation}\label{norm-Ha}
  \| \phi \|^2_{\mathcal{H}^{\pm}_\alpha}=\int_{I}\! \left( \sup_{\xi, \eta}  \beta^{\pm}_1(\eta)
    \beta_2(\xi) |\phi (\xi, \eta, x)|\right)^2 dx.
\end{equation}
They form a  Gelfand triplet
$$
\mathcal{H}^+_\alpha  \subset \mathcal{H}^+_0 = \mathcal{H}^-_0 \subset \mathcal{H}^-_\alpha,
$$
which will be used in Section~\ref{sec.resolvent}.
Below, the function spaces in \eqref{spaces} are used for $\alpha\in\{0, a\}$, where $a$ is the same as in \eqref{g-decay}.
Recall that $\hat g\in\mathcal{X}^+_a$ for certain $0<a<1$ (cf.~Assumption~\ref{as.decay}).

Now we can formally discuss the operators involved in the linearized problem \eqref{linear},
\eqref{linear-2}. First, we note that $\bW$  defined in \eqref{def-W}
can be viewed as the Fredholm integral operator on $L^2(I)$.
It is a compact self-adjoint operator. Therefore, the eigenvalues of $\bW$ are real with the
only accumulation point at $0$. We denote the set of eigenvalues of $\bW$ by
$\sigma _p(\bW)$.

 Next, we turn to operators $\bL_j$ and $\bB$.
 It follows that they are densely defined on $\mathcal{H}^+_0$. 
In addition, $\bB$ is a bounded operator as shown in the following lemma.
\begin{lemma}
  $\bB$ is a bounded operator on $\mathcal{H}^+_0$.
\end{lemma}
\begin{proof}
When $\alpha =0$, $\beta ^{\pm}_1 (\eta) = 1$. 
By the Cauchy--Schwarz inequality, we have
\begin{eqnarray*}
|| \bB[\phi] ||^2_{\mathcal{H}^+_0}
&=& \int_{I} \left( \sup_{\xi, \eta} \left\{ \beta_2(\xi)(1-e^{-\xi})|\hat{g}(\eta)|\right\}  
\Bigl| \int_I W(x,y)\phi (0,0,y)dy  \Bigr| \right)^2 dx \\
& \leq & \int_{I^2} |W(x,y)|^2dxdy\cdot \sup_{\eta}\left\{|\hat{g}(\eta)|^2\right\}
\sup_{\xi, \eta} \left\{\beta_2(\xi)^2 (1-e^{-\xi})^2 \right\}\int_I |\phi (0,0,x)|^2 dx \\
&=& || \bW ||^2_{L^2} \cdot \sup_{\eta}\left\{|\hat{g}(\eta)|^2\right\}
\sup_{\xi, \eta} \left\{\beta_2(\xi)^2 (1-e^{-\xi})^2 \right\}\int_I (\beta_2 (0)|\phi (0,0,x)|)^2 dx \\
&\leq & || \bW ||^2_{L^2} \cdot \sup_{\eta}\left\{|\hat{g}(\eta)|^2\right\}
\sup_{\xi, \eta}\left\{ \beta_2(\xi)^2 (1-e^{-\xi})^2 \int_I (\beta_2 (\xi)|\phi (\xi,\eta ,x)|)^2 dx \right\}\\
&=& || \bW ||^2_{L^2} \cdot \sup_{\eta}|\hat{g}(\eta)|^2\cdot
\sup_{\xi} \beta_2(\xi)^2 (1-e^{-\xi})^2 || \phi ||^2_{\mathcal{H}^+_0}. 
\end{eqnarray*}
For both cases $\xi \geq 0$ and $\xi \leq 0$, we have $\sup_{\xi} \beta_2(\xi)^2 (1-e^{-\xi})^2 =1$, which proves the lemma.
\end{proof}

Since $\bL_j$ is closed and $\bB$ is bounded, $\bS = \bL_1 + K\bB$ is also a closed operator on $\mathcal{H}^+_0$.

%%%%%%%%%%%%%%%%%%%%%%%%%%%%%%%%%%%%%%%%%%%%%%%%%%%%%%%%%%%%%%%%%%

\subsection{The spectrum of \texorpdfstring{$\bS$}{TEXT}}\lbl{sec.spectrum}

In this section, we establish several basic facts about the spectra of $\bL_j$ and $\bS$. In particular, we
describe the residual spectrum and derive a transcendental equation for the eigenvalues of $\bS$. Throughout this section, we view $\bL_j$ and $\bS$ as operators densely defined on $\cH^+_0$
(cf.~\eqref{spaces}-\eqref{norm-Ha}).

Let $j\in\Z$ be fixed and consider
\begin{equation}\label{begin-res}
(\lambda -\bL_j)u=v.
\end{equation}
Applying the Fourier transform to both sides of \eqref{begin-res} and using \eqref{Lj},
we have 
$$
\left(\lambda -\iu\zeta -\iu j\omega\right) \cF [u] =\cF [v],
$$
where
\begin{equation}\label{F-transform}
\cF[v](\zeta, \omega ,x)=\frac{1}{(2\pi)^2} \int_{\R^2}\! e^{-\iu(\xi \zeta + \eta \omega )} v(\xi, \eta,x) d\xi d\eta. 
\end{equation}
Thus,
\begin{equation}\label{L-res}
  u=(\lambda- \bL_j)^{-1} v = \cF^{-1} \cF[u] \\
   =
    \int_{\R^2} \frac{e^{\iu (\xi \zeta + \eta \omega )}}{\lambda - \iu \zeta - \iu j\omega }
    \cF[v](\zeta, \omega ,x) d\zeta d\omega.
    \end{equation}
    Next we use \eqref{L-res} to locate the residual spectrum of $\bL_j$.
    \begin{lemma}\label{lem.residual} For arbitrary $j\in \Z,$ 
     the residual spectrum of $\bL_j$ is given by
      $$
      \sigma (\bL_j) = \{ z\in\C:\; -1\le \Re z\le 0 \}
      \footnote{Here and below, we use $\Re z$ and $\Im z$ to denote the real and imaginary
      parts of $z\in\mathbb{C}$.}.
      $$
    \end{lemma}
    \begin{proof}
      The proof is based on \eqref{L-res} and the following identity for $\Re~\lambda \neq 0$
      \begin{equation}\label{express-Cauchy}
        \frac{1}{\lambda -\iu\zeta -\iu j\omega} =
        \left\{
          \begin{array}{ll}
\displaystyle \int^{\infty}_{0} e^{-(\lambda  -\iu\zeta -\iu j\omega )t}dt,  & \;  \Re~\lambda  > 0, \\
\displaystyle -\int_{-\infty}^0  e^{-(\lambda  -\iu\zeta -\iu j\omega )t}dt, & \;\Re~\lambda  < 0, \\
                \end{array}
              \right.
\end{equation}

By plugging \eqref{express-Cauchy} into \eqref{L-res}, we have
\begin{equation}\label{rewrite-res}
(\lambda - \bL_j)^{-1}[v](\xi, \eta, x ) = \left\{ \begin{array}{ll}
\displaystyle \int^{\infty}_{0}\! e^{-\lambda t} v (\xi + t, \eta + jt,x)dt, & \;\Re~\lambda  >0,\\
\displaystyle  -\int^{0}_{-\infty} \! e^{-\lambda t} v (\xi + t, \eta + jt,x)dt & \;\Re~\lambda <0.
                                                          \end{array}
                                                        \right.
\end{equation}
The right--hand side of \eqref{rewrite-res} belongs to $\mathcal{H}^+_0$ for any
$v\in \mathcal{H}^+_0$ only if $\Re~\lambda < -1$ or $\Re~\lambda>0$.
For $ -1\leq  \Re~\lambda\ \leq 0$, the set of $v$ such that the right-hand side exists is not dense in $\mathcal{H}^+_0$.The statement of the lemma follows.
\end{proof}

Next, we turn to the eigenvalue problem for $\bS$ (cf.~\eqref{linear})
\begin{equation}\label{recall-S}
  \lambda v=(\bL_1 +K\bB) v.
\end{equation}

\begin{lemma}\label{lem.eig-S}
  Let $\mu$ be a nonzero eigenvalue of $\bW$ and let $V\in L^2(I)$ be a corresponding
  eigenfunction. Define
  \be\lbl{define-D}
D(\lambda ; \xi, \eta)=\int_{\R} \left( \frac{1}{\lambda -\iu\omega } -
        \frac{e^{-\xi}}{\lambda +1-\iu\omega } \right)
      e^{\iu\eta \omega } g(\omega)d\omega,
      \ee
and
  \begin{equation}\label{def-D}
    D(\lambda):=D(\lambda; 0,0) = 
 \int_{\R} \left( \frac{1}{\lambda -\iu\omega } -
  \frac{1}{\lambda +1-\iu\omega } \right) g(\omega )d\omega.
\end{equation}
Then the root $\lambda=\lambda(\mu)$ of 
the following equation 
\begin{equation}\label{D-eqn}
D(\lambda ) = \frac{1}{K\mu},
\end{equation}
not belonging 
to $\p\sigma(\bL_1)=\{z\in\C: \Re~z=-1 \;\mbox{or}\; \Re~z=0\}$,
is an eigenvalue of $\bS$ on $\cH^+_0$.
For each such root $\lambda=\lambda(\mu)$
the corresponding eigenfunction is given by
\begin{equation} \label{eigen-v}
v(\xi, \eta ,x) 
= D(\lambda ; \xi, \eta) V(x).
\end{equation}
\end{lemma}
\begin{proof}
From \eqref{recall-S} we obtain
\begin{equation}\label{eigen-S}
v = K(\lambda -\bL_1)^{-1} (1-e^{-\xi})\hat{g}(\eta) \bW[v(0,0, \cdot)].
\end{equation}
By \eqref{L-res},
\begin{equation}\label{compute-D}
  \begin{split}
    (\lambda -\bL_1)^{-1}(1- e^{-\xi})\hat{g}(\eta)
    &= \int_{\R^2} \frac{e^{\iu (\xi \zeta + \eta \omega )}}{\lambda -\iu(\zeta +\omega) } 
      \cF[(1-e^{-\xi})\hat{g}(\eta)]d\zeta d\omega  \\
&=  \int_{\R^2} \frac{e^{\iu (\xi \zeta + \eta \omega )}}{\lambda -\iu(\zeta +\omega) } 
(\delta (\zeta) - \delta (\zeta - \iu)) g(\omega ) d\zeta d\omega \\
      &=D(\lambda ; \xi, \eta),
\end{split}
\end{equation}
where we used \eqref{define-D} in the last line.
The combination of \eqref{eigen-S} and \eqref{compute-D} yields
\be\lbl{next-step-D}
v = KD(\lambda ; \xi, \eta)\bW[v(0,0, \cdot)].
\ee
By plugging $\xi=\eta=0$ in \eqref{next-step-D}, we have 
\begin{equation}\label{plug-xi-eta}
v(0,0,x) = KD(\lambda) \bW[v(0,0,\cdot)](x).
\end{equation}
If $\lambda=\lambda(\mu)$ solves \eqref{D-eqn}, then \eqref{plug-xi-eta}
holds for
\begin{equation}\label{use-V}
v(0,0,x)=V(x).
\end{equation}
This shows that the set of roots of \eqref{D-eqn} for $\mu\in\sigma_p(\bW)$
coincides with the set of eigenvalues of $\bS$.
The corresponding eigenfunctions are found from \eqref{next-step-D}:
\begin{equation}\label{S-v}
v=K\mu D(\lambda(\mu); \xi, \eta) V,
\end{equation}
which coincides with \eqref{eigen-v} up to a multiplicative constant.
It remains to show that $v\in\mathcal{H}^+_0$ when $\lambda \notin \partial \sigma(\bL_1)$.
To this end, note that (\ref{express-Cauchy}) is applicable to (\ref{define-D}) when $\Re~\lambda \neq 0,1$ as
\begin{equation*}
D(\lambda ;\xi, \eta)= \left\{ \begin{array}{ll}
\int^{\infty}_{0}\! e^{-\lambda t} (1-e^{-(\xi + t)}) \hat{g}(\eta + t) dt,  & \mbox{if}\;\Re~\lambda > 0,\\
\int^{-\infty}_{0}\! e^{-\lambda t} \hat{g}(\eta + t) dt -
   \int^{\infty}_{0}\! e^{-\lambda t} e^{-(\xi + t)} \hat{g}(\eta + t) dt, &\mbox{if}\;  -1<\Re~\lambda < 0,\\
\int^{-\infty}_{0}\! e^{-\lambda t} (1-e^{-(\xi + t)}) \hat{g}(\eta + t) dt, &\mbox{if}\; \Re~\lambda <-1. 
                               \end{array}
                             \right.
\end{equation*}
This shows that $D(\lambda ;\xi, \eta)\in \mathcal{Y}^+_0$ and $v\in \mathcal{H}^+_0$
if $\hat{g} \in \mathcal{X}^+_0$ and  $\lambda \notin \partial \sigma (\bL_1)$.
\end{proof}

%%%%%%%%%%%%%%%%%%%%%%%%%%%%%%%%%%%%%%%%%%%%%%%%%%%%%%%%%%%%%%%%%%%

\subsection{The eigenvalue equation}\label{section3.3}

In this subsection, we study  equation \eqref{D-eqn}, whose roots yield the eigenvalues
of $\bS$.

Note that $D(\lambda; \xi, \eta)$ is a holomorphic function on 
$\mathbb{H}_0:=\{\lambda \in \C:\; \Re~\lambda>0\}$. Using Assumption~\ref{as.decay},
$D(\lambda; \cdot)$ can be extended analytically to $\mathbb{H}_a:=\{\lambda \in \C:\; \Re~\lambda>-a\},\,\, 0<a<1$.
Indeed, by rewriting the integrand in the definition of $D(\lambda; \xi, \eta)$
\begin{equation*}
  \begin{split}
D(\lambda ;\xi, \eta)&= \int^{\infty}_{0}\! e^{-\lambda t} (1-e^{-(\xi + t)}) \hat{g}(\eta + t) dt\\
&= \int^{\infty}_{0}\! e^{-(\lambda + a) t} e^{-a\eta}(1-e^{-(\xi + t)}) e^{a(\eta + t)}\hat{g}(\eta + t) dt,
\end{split}
\end{equation*}
one can see that this integral exists and defines an analytic function on $\mathbb{H}_{a}$,
because $\hat g\in \mathcal{X}^+_a$ (cf.~Assumption~\ref{as.decay}).
Moreover,  by applying Sokhotski--Plemelj formula \cite{Simon-Complex} to (\ref{define-D}), the analytic continuation is also written by
\begin{equation}\label{Dcont}
\mathcal{D}(\lambda ; \xi, \eta)
=\left\{\begin{array}{ll}
         \displaystyle
\int_{\mathbb{R}} \left( \frac{1}{\lambda -\iu\omega } - \frac{e^{-\xi}}{\lambda +1-\iu\omega } \right)
          e^{\iu\eta \omega } g(\omega )d\omega\, (=D(\lambda ;\xi, \eta)), & \Re\lambda >0,\\
 \displaystyle \lim_{x\to0+}
\int_{\mathbb{R}} \left( \frac{1}{x+\iu(y-\omega) } - \frac{e^{-\xi}}{ 1+\iu(y-\omega) } \right)
          e^{\iu\eta \omega } g(\omega )d\omega,
         & \lambda=\iu y,\\          
\displaystyle
\int_{\mathbb{R}} \left( \frac{1}{\lambda -\iu\omega } - \frac{e^{-\xi}}{\lambda +1-\iu\omega } \right)
          e^{\iu\eta \omega } g(\omega )d\omega + 2\pi e^{\eta \lambda } g(-\iu\lambda ), & -a<\Re\lambda<0.\
                                                                                        \end{array}\right.
                                                                                    \end{equation}
From these expressions, it turns out that for $-a<\Re\lambda<0$,  $\cD(\lambda;\cdot)\in\mathcal{Y}^-_a$,
while $\cD(\lambda;\cdot)\in\mathcal{Y}^+_a$ for $\Re\lambda>0$.                                                                                    

Denote $\cD(\lambda)=\cD(\lambda;0,0)$.  $\cD(\lambda)$ provides an analytic continuation
of $D(\lambda)$ to $\mathbb{H}_a$.

\begin{definition}\label{def.generalizedEV}
  Let $\mu$ be a nonzero eigenvalue of $\bW$ and let
  $V$ denote the corresponding eigenfunction. Suppose $\lambda\in\mathbb{H}_a$ is
  a root of the following equation
\begin{equation}
\mathcal{D}(\lambda ) = \frac{1}{K\mu}.
\label{g-ev}
\end{equation}
Then $\lambda$ is called a generalized eigenvalue of $\bS$.
The corresponding generalized eigenfunction is given by
\[ v=\mathcal{D}(\lambda;\xi,\mu)V(x) \in \cH_a^-.\]
\end{definition}

\begin{remark}
  Generalized eigenvalues coincide with regular eigenvalues on $\mathbb{H}_0$.
The reason why the roots of \eqref{g-ev} are called generalized eigenvalues
of $\bS$ will become clear in Section~\ref{sec.resolvent}.
See \cite{Chi15b} for the generalized spectral theory.
\end{remark}

The following theorem describes solutions of \eqref{g-ev} for even unimodal $g$.

% \begin{assumption}\lbl{as.symmetry}
% Let $g$ be an even symmetric unimodal function.
% \end{assumption}

\begin{theorem}\label{thm.crossing}
  Suppose $g$ is an even unimodal probability density function and denote
  \be\lbl{criticalK}
  K_c=(\mu_{max} g_0)^{-1},
  \ee
  where
\begin{equation}\label{g0}
  g_0:=\pi g(0) -\int_\R \frac{g(s)}{1+s^2}ds
\end{equation}
and $\mu_{max}$ is the largest (positive) eigenvalue of $\bW$.

Under Assumption~\ref{as.decay}, the following holds.
  \begin{enumerate}
  \item
    For $K\in [0,K_c)$ there are no generalized eigenvalues 
    with positive real part.
  \item For sufficiently small $\epsilon>0$, there is a real generalized eigenvalue
    $\lambda=\lambda(K)$ for $K\in (K_c-\epsilon, \infty)$. In addition,
    $$
    \lambda(K_c)=0\quad\mbox{and}\quad  \lambda^\prime(K_c)>0.
    $$
  \end{enumerate}
  \end{theorem}
  \begin{remark}
    Using the identity $\pi^{-1}=\int_\R (1+s^2)^{-1}ds$, one can see that $g_0$ defined
    in \eqref{g0} is positive:
\begin{equation}\label{g>0}
 g_0=\int_\R \frac{g(0)-g(s)}{1+s^2}ds>0,
\end{equation}
\end{remark}
\begin{remark}\lbl{rem.generalized}
 The positive generalized eigenvalue
  $\lambda(K), K>K_c,$ identified in the theorem is an eigenvalue of $\bS$. At $K=K_c+0$ it
  hits the residual spectrum of $\bS$. There are no eigenvalues of $\bS$ for $K\le K_c$. However,
  as a generalized eigenvalue, $\lambda(K)$ is well defined for
  $K_c-\epsilon<K\le K_c$. It crosses the imaginary
  axis transversally at $K=K_c$.
  \end{remark}

  The proof of the theorem relies on the following observations.
  \begin{lemma}\label{lem.D}
  Let $g$ be an even unimodal probability density function (Figure~\ref{f.1}\textbf{a}).
  Then 
    \begin{eqnarray}
      \label{D-1}
      \lim_{x\to 0+}  \cD (x+\iu y)&=& \pi g(y) -\iu \cdot \pv\int_\R \frac{g(s)}{y-s}ds -
                                    \int_\R \frac{g(s)}{1+\iu (y-s)} ds,\\
     \label{D-2}
      \lim_{y\to \pm\infty}\lim_{x\to 0+}  \cD (x+\iu y)&=&0,\\
      \label{D-3}
      \cD^\prime(0)&<&0.
    \end{eqnarray}
    Here, $\pv$ stands for the Cauchy principal value of an improper integral.
    In addition, $\mathcal{C}=\cD(\iu\R)$ is a bounded closed curve which intersects 
    the positive real semiaxis
    at a unique point $(g_0,0)$ (see Figure~\ref{f.1}\textbf{b}).
\end{lemma}
We postpone the proof of Lemma~\ref{lem.D} until after the proof of the theorem.

\begin{figure}
  \centering
 \textbf{a} \includegraphics[width=.47\textwidth]{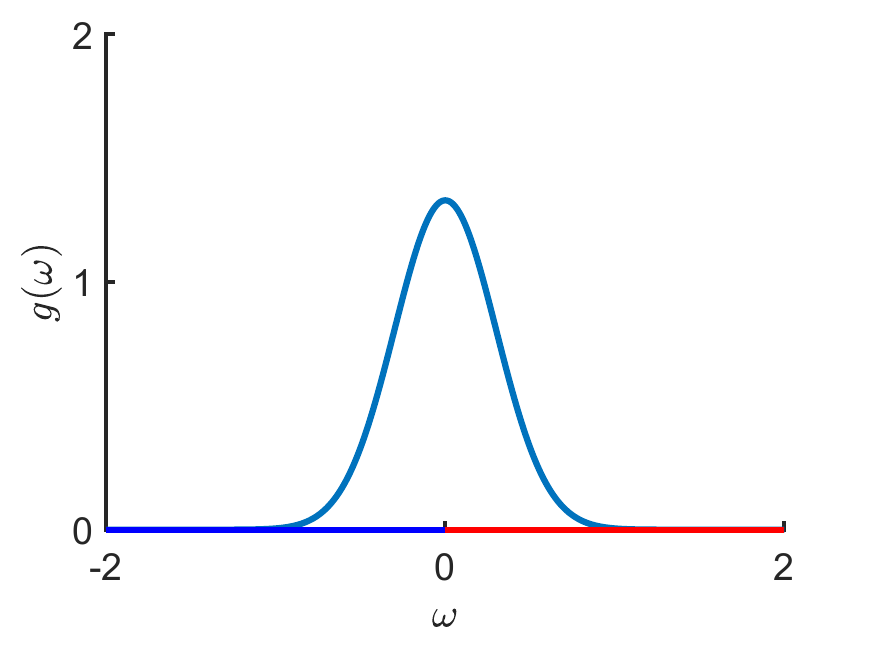}
 \textbf{b}  \includegraphics[width=.47\textwidth]{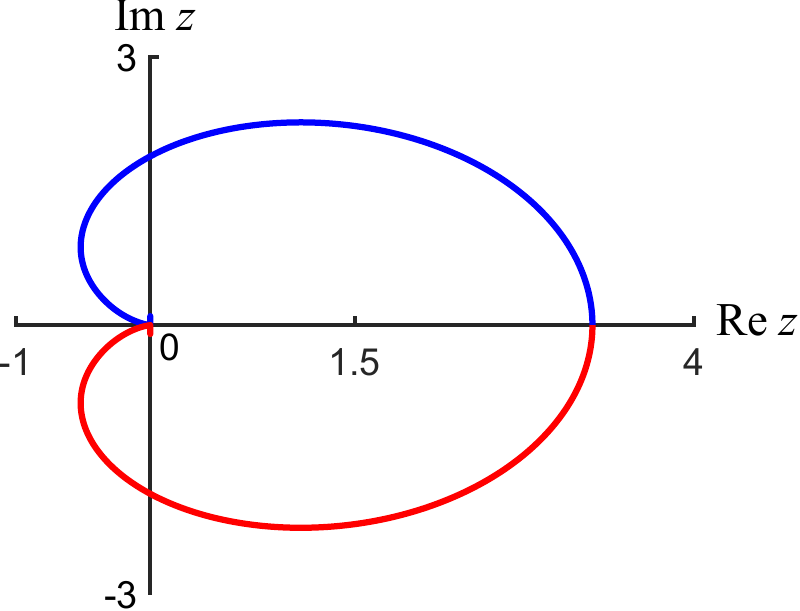}
 % \textbf{c} \includegraphics[width=.3\textwidth]{FFF/uni-bif.pdf}
  \caption{A unimodal probability density function $g$ (\textbf{a}) and
    the corresponding critical curve (\textbf{b}).} 
             	\label{f.1}
             \end{figure}

\begin{proof}[Proof of Theorem~\ref{thm.crossing}]
  Let $\mu>0$ be an eigenvalue of $\bW$.
  Note that $\mathcal{C}=\cD(\iu \R)$ is a bounded closed curve
  intersecting positive real semiaxis at a single point $(g_0,0)$ by
  Lemma \ref{lem.D}. Since $\cD (\lambda)$ is holomorphic in
  $\mathbb{H}_0$,  by the Argument Principle, the eigenequation
  \begin{equation}\label{EVeqn}
  \cD(\lambda) =(K\mu)^{-1}
  \end{equation}
  has a root in $\mathbb{H}_0$ if the winding number of $\cD$ about $(K\mu)^{-1}$ is positive. Therefore,
  \eqref{EVeqn} has no roots in $\mathbb{H}_0$ for $0<K\mu<g_0^{-1}$ for any $\mu \in \sigma_p (\bW)\backslash \{0\}$, i.e., for $K\in [0,K_c)$ with $K_c=(\mu_{max} g_0)^{-1}$.
  Since $g(\omega)$ is an even function, $\cD(\iu \R)$ intersects positive real semiaxis when $\lambda = 0$. This implies that $g_0 = \lambda (0)$, which is obtained from (\ref{D-1}) as 
  \begin{eqnarray*}
\lim_{y\to 0} \lim_{x\to 0+}  \cD (x+ \iu y)&=& \pi g(0) +\iu \cdot \pv \int_\R \frac{g(s)}{s}ds -
                                    \int_\R \frac{1+\iu s}{1+s^2}g(s) ds \\
&=& \pi g(0)-\int_\R \frac{1}{1+s^2}g(s) ds.
\end{eqnarray*}
  This proves the first statement of the theorem with formulae (\ref{criticalK}) and (\ref{g0}). 
  
  Since $\cD^\prime(0)\neq 0$ by Lemma \ref{lem.D} and $\cD$ is holomorphic at $0$, it is locally invertible, and the inverse
  is a holomorphic function. By differentiating both sides of \eqref{EVeqn} with respect to $K$, we
  have
  $$
  \lambda^\prime(K_c)=\frac{-1}{\mu_{\max}K^2_c \cD^\prime (0)}>0.
  $$
  Since $\mu_{max}$ is an isolated eigenvalue of $\bW$, there are no other roots of \eqref{EVeqn}
  for $K\in (K_c-\epsilon, K_c+\epsilon)$.
 \end{proof}

\begin{proof}[Proof of Lemma~\ref{lem.D}]
Equation~\eqref{D-1} follows from the Sokhotski--Plemelj formula \cite{Simon-Complex}.
Equation~\eqref{D-2} follows from \eqref{D-1}, because $g\in L^1(\R)$.
 To show \eqref{D-3}, note first
  \be\lbl{differentiateD}
  \begin{split}
    \cD^\prime(\lambda ) &= -\int_\R \left( \frac{1}{(\lambda -\iu s)^2 } 
                  - \frac{1}{(\lambda +1 -\iu s )^2 } \right)  g(s ) ds \\
                &=-\iu \int_\R \left( \frac{1}{\lambda -\iu s} - \frac{1}{\lambda +1 -\iu s } \right)
                g^\prime(s) ds.
              \end{split}
              \ee
              Using the Sokhotski--Plemelj formula again, from \eqref{differentiateD} we get
              \be\lbl{SPagain}
\lim_{\lambda \to 0+}D^\prime(\lambda )
= \pv\int_\R \frac{g^\prime(s)}{s} ds  + \iu\left(
\int_\R\frac{1}{1 -\iu s } g^\prime(s) ds  - \pi g^\prime (0)\right).
\ee

Since $g^\prime$ is an odd function, we have
\begin{equation*}
  \begin{split}
\lim_{\lambda \to 0+}D^\prime(\lambda )
&=\lim_{\varepsilon \to 0+} \int^\infty_\varepsilon \frac{g^\prime(s) - g^\prime(-s)}{s} ds 
-  \int_\R\frac{s g^\prime(s) }{1+s^2 } ds \\
&= 2 \lim_{\varepsilon \to 0+} \int^\infty_\varepsilon \frac{g^\prime(s)}{s }ds 
- 2 \int^\infty_0\frac{s g^\prime(s)}{1+s^2 } ds \\
&= 2 \lim_{\varepsilon \to 0+} \int^\infty_\varepsilon \frac{g^\prime(s)}{s(1+s^2)} ds<0,
\end{split}
\end{equation*}
where we used that $g^\prime(s)\le 0$ for $s>0$ and is not everywhere zero on $\R^+$ to obtain the
last inequality.

Finally, \eqref{D-2} implies that $\mathcal{C}$ is a bounded closed curve. From \eqref{D-1},
for $\mathcal{G}(y)=\lim_{x \to 0+}\cD(x+\iu y )$ we have
\begin{eqnarray}
  \label{G-re}
  \Re\mathcal{G}(y)&=&\pi g(y)-\int_\R \frac{g(s)}{1+(y-s)^2}ds,\\
  \label{G-im}
  \Im\mathcal{G}(y)&=& \int_\R \frac{(y-s)g(s)}{1+(y-s)^2}ds -\pv\int_\R \frac{g(s)}{y-s}ds.
\end{eqnarray}
Here, $\Re z$ and $\Im z$ stand for the real and imaginary parts of $z\in\C$ respectively.
Using  even symmetry of $g$, from \eqref{G-re} and \eqref{G-im} it follows that $(g_0,0)$ is
a point of intersection of $\mathcal{C}$ with the positive semiaxis (Figure~\ref{f.1}\textbf{b}). Even symmetry and unimodality of
$g$ provides the uniqueness:
\be\lbl{uniqueness}
\begin{split}
-\Im D(x+\iu y) &= \int_\R \left( \frac{y-\omega}{x^2+(y-\omega)^2} -
  \frac{y-\omega}{(x+1)^2+(y-\omega)^2} \right) g(\omega) d\omega\\
&=- \int_\R
\left(
  \frac{\omega}{x^2+\omega^2} -
  \frac{\omega}{(x+1)^2+\omega^2}
\right) g(y+\omega) d\omega\\
&=- \left\{ \int_{-\infty}^0 +\int_0^\infty \right\}
\left(
\frac{\omega}{x^2+\omega^2} -
\frac{\omega}{(x+1)^2+\omega^2}\right) 
g(y+\omega) d\omega\\
&= -\int_0^\infty \left(
  \frac{\omega}{x^2+\omega^2} -
  \frac{\omega}{(x+1)^2+\omega^2}
\right) g(y+\omega) d\omega\\
&+
\int_0^\infty \left(
  \frac{\omega}{x^2+\omega^2} - \frac{\omega}{(x+1)^2+\omega^2}
\right) g(y-\omega) d\omega\\
&=-\int_0^\infty 
\frac{\omega \left( (x+1)^2-x^2\right)}{(x^2+\omega^2)\left((x+1)^2 +\omega^2\right)}
\left( g(y+\omega)-g(y-\omega) \right)d\omega.
\end{split}
\ee
Even symmetry of $g$ implies that $y=0$ solves $\Im D(x+\iu y)=0$ (cf. \eqref{uniqueness}).
For unimodal $g$ this solution is unique, because in this case 
$
g(y+\omega)-g(y-\omega)>0
$
when $y<0, \; \omega >0$ and $g(y+\omega)-g(y-\omega)<0
$
when $y>0, \; \omega >0$.

\end{proof}

\subsection{The resolvent and the Riesz projector}\label{sec.resolvent}

In this section, we compute the resolvent operator
\begin{equation*}
  \bR_\lambda=\left(\lambda -\bS\right)^{-1}
\end{equation*}
and the Riesz projector for $\bS$.

\begin{lemma} For $\phi\in\cH_0^+$, 
  \begin{equation}\label{expressR}
    \bR_\lambda\phi =
    \left(\lambda -\bL_1\right)^{-1} \phi + K  D(\lambda;\xi,\eta) 
    \left(\bI-KD(\lambda)\bW\right)^{-1}
  \bW\left[ \left(\left(\lambda-\bL_1\right)^{-1}\phi \right)(0,0,\cdot)\right].
 \end{equation}
 \end{lemma} 
\begin{proof}
From the definitions of $\bR_\lambda$ and $\bS$ we have
\begin{equation}\label{def-R}
  \left(\lambda -\bL_1-K\bB\right)\bR_\lambda \phi =\phi\quad \forall
  \phi\in\cH_0^+.
\end{equation}
We rewrite \eqref{def-R} as
\begin{equation*}
  \left(\lambda -\bL_1\right) \left(\bI-K \left(\lambda -\bL_1\right)^{-1}\bB\right)
  \bR_\lambda \phi =\phi.
\end{equation*}
Thus, \eqref{compute-D} gives
\begin{alignat}{1}\nonumber
  \bR_\lambda \phi &= \left(\lambda -\bL_1\right)^{-1} \phi + K \left(\lambda -\bL_1\right)^{-1}
  \bB\bR_\lambda\phi\\
  \label{R-identity}
  &=\left(\lambda -\bL_1\right)^{-1} \phi + K
  D(\lambda;\xi,\eta) \bW\left[(\bR_\lambda \phi)(0,0,\cdot)\right].
\end{alignat}
By setting $\xi=\eta=0$ and applying $\bW$ to both sides of \eqref{R-identity}, after
some algebra we obtain
\begin{equation}\label{after-algebra}
  \bW\left[ (\bR_\lambda\phi)(0,0,\cdot)\right]=\left(\bI-KD(\lambda)\bW\right)^{-1}
  \bW\left[ \left(\left(\lambda-\bL_1\right)^{-1}\phi \right)(0,0,\cdot)\right].
\end{equation}
By plugging \eqref{after-algebra} into \eqref{R-identity}, we obtain \eqref{expressR}.
\end{proof}

From \eqref{expressR}  one can see that the spectrum of $\bS$, as an operator on $\cH^+_0$,
is the union of the residual spectrum of $\bL_1$ and point spectrum of $\bS$
coming from the factor $\left(\bI-KD(\lambda)\bW\right)^{-1}$;
$$
\sigma (\bS)= \sigma (\bL_1) \bigcup \sigma_p (\bS).
$$
Although $\{ z\in\C:\; -1\le \Re z\le 0 \}$ is a residual spectrum of $\bL_1$, as shown in Section \ref{section3.3}, $D(\lambda; \xi,\eta)$ and $D(\lambda)$ have the analytic continuations from $\bbH_0$ to $\bbH_a$ denoted by $\cD(\lambda; \xi,\eta)$ and $\cD(\lambda)$, respectively.
Recall that $\cD(\lambda; \xi,\eta) \in \mathcal{Y}_a^-$ when $\hat{g} \in \mathcal{X}^+_a$.
In a similar manner, we can verify that $\left(\lambda-\bL_1\right)^{-1}$ has an analytic continuation from $\bbH_0$ to $\bbH_a$ as an operator from $\mathcal{H}_a^+$ to $\mathcal{H}_a^-$.
Using them, we define the analytic continuation of $\bR_\lambda$ from $\bbH_0$ to $\bbH_a$ by
\begin{equation}\label{genR}
    \cR_\lambda\phi =
    \left(\lambda -\bL_1\right)^{-1} \phi + K  \cD(\lambda;\xi,\eta) 
    \left(\bI-K\cD(\lambda)\bW\right)^{-1}
  \bW\left[ \left(\left(\lambda-\bL_1\right)^{-1}\phi \right)(0,0,\cdot)\right].
 \end{equation}
 This is an operator from $\mathcal{H}_a^+$ into $\mathcal{H}_a^-$ called the generalized resolvent.
 Note that the generalized eigenvalues in Definition~\ref{def.generalizedEV}
 are the poles of $\cR_\lambda$.
 Recall the definitions of the Banach spaces $\mathcal{H}^\pm_a$ and
$\mathcal{H}^\pm_0$ (cf.~\eqref{spaces}) and note that they form Gelfand triplet:
$$
\mathcal{H}^+_a \subset \mathcal{H}^+_0 = \mathcal{H}^-_0 \subset \mathcal{H}^-_a.
$$
For the analytic continuation, the domain of $\cR_\lambda$ is restricted to $\cH^+_a$
and the range is extended to $\cH^-_a$.

  The corresponding generalized Riesz projection  $\mathcal{P}_\lambda:\cH^+_a\rightarrow \cH^-_a$
  is given by
  \be\lbl{gen-Riesz}
  \mathcal{P}_\lambda \phi =\frac{1}{2\pi\iu} \oint_{c} \cR_z \phi dz,
  \ee
  where $c$ is a positively oriented closed curve, which encircles an eigenvalue $\lambda$
  but no other generalized eigenvalues. Note that the generalized eigenfunction corresponding
  to $\lambda$ in Definition~\ref{def.generalizedEV} is an element of the image
  $\mathcal{P}_\lambda$.
We conclude this section with the computation of the generalized Riesz projection 
for the bifurcating eigenvalue $\lambda=0$:
\be\lbl{RieszS}
\cP_0 [\phi]=\frac{1}{2\pi\iu} \oint_c \cR_z \phi dz,
\ee
where $c$ is a positively oriented contour around the origin, which does not contain
any other generalized eigenvalues of $\bS$.
  
Here and in the remainder of this paper, we suppress the subscript of $\cP_0$ since
from now on we will only be interested in the bifurcating eigenvalue.
Further, suppose that $\mu_{max}$, the largest positive eigenvalue of $\bW$, is simple.
The Riesz projection for $\bW$ corresponding to $\mu_{max}$ is denoted by
\be\lbl{RieszW}
\bP[\Phi]=\frac{1}{2\pi\iu} \oint_C \left( z-\bW\right)^{-1} \Phi dz,
\ee
where $C$ is a positively oriented contour around $\mu_{max}$, which does not contain
any other eigenvalues of $\bW$.

\begin{lemma}\lbl{lem.R}
  \be\lbl{compute-cP}
  \cP[\phi](\xi,\eta,x)=\rho_1 \cD(0 ;\xi,\eta)\bP[\Phi] (0,0,x),
  \ee
  where $\Phi:=\lim_{\lambda\to 0+}\left(\lambda -\bL_1\right)^{-1}\phi$ and $\rho_1>0$.
  \end{lemma}
  \begin{proof}
    First, note
\begin{equation}\label{generalized-resS}
      \cR_\lambda [\phi](\xi,\eta, x) =\Phi_\lambda(\xi,\eta,x)
      +K\mathcal{D}(\lambda; \xi,\eta) \left( \bI -K \mathcal{D}(\lambda) \bW\right)^{-1} \bW\left[
        \Phi_\lambda (0,0,\cdot)\right](x),
  \end{equation}
  where
  \begin{equation}\label{def-Phi-lambda}
    \Phi_\lambda:=\left(\lambda -\bL_1\right)^{-1}\phi.
    \end{equation}
  
    Using \eqref{generalized-resS}, we have
    \be\lbl{new-resS}
      \oint_c \cR_z [\phi](\xi,\eta,x) dz = 
      K\oint_c\cD(z;\xi,\eta) \left(\bI-K\cD(z) \bW\right)^{-1}
      \bW[\Phi_z(0,0,\cdot)](x)dz,
      \ee
      where we used $\oint_c \Phi_z(\xi,\eta,x) dz =0,$ because
      $\Phi_z(\xi,\eta,x)$ is regular inside $c$. By continuity $\cD^\prime(z)\neq 0$
      on $c$ provided $c$ is sufficiently small. 
Change variable in \eqref{new-resS}
      to $z:=z(\alpha)$
      such that
      \be\lbl{change-var}
      \alpha=\left(K\cD(z)\right)^{-1}.
      \ee
Note that 
      \be\lbl{change-var2}
      \mu_{max}=\left(K_c\cD(0)\right)^{-1},
      \ee
as follows from \eqref{g-ev}.
    %  and using \eqref{rootD},
      We have
      \be\lbl{comp-resolvent}
      \begin{split}
        \oint_c\cR_z [\phi](\xi,\eta,x) dz &= 
        - \displaystyle \oint_C \frac{\cD\left(z(\alpha); \xi, \eta \right)}
        {\alpha^{2} \cD^\prime \left(z(\alpha)\right)}
\bW(\bI - \alpha^{-1}\bW )^{-1} 
 [\Phi_{z(\alpha)}] (0, 0, x) d\alpha\\
 &= -\oint_C \frac{\cD(z(\alpha); \xi, \eta )}
 {\alpha  \cD^\prime \left(z(\alpha)\right)} \bW ( \alpha - \bW )^{-1} 
    \left[ \Phi_{z(\alpha)}\right](0, 0,x)  d\alpha, 
  \end{split}
  \ee
  where $C$ stands for the image of $c$ under \eqref{change-var}.
  Since $\mu_{max}$ is a simple eigenvalue of $\bW$, $\left(\alpha-\bW\right)^{-1}$ has
  a simple pole $\alpha = \mu_{max}$ inside $C$ while the other factor under the integral is regular.
  Thus, \eqref{comp-resolvent} yields
  \be\lbl{projection-continue}
    \oint_c \cR_z [\phi](\xi,\eta,x) dz =
    \frac{-\cD(0;\xi,\eta)}{\mu_{max}\cD^\prime(0)} \bW
    \oint_C \left(\alpha-\bW\right)^{-1} \left[ \Phi \right](0,0,x) d\alpha.
    \ee
    Noting that the integral on the right--hand side implements Riesz projection of $\bW$ for $\mu_{max}$,
    we further have

 \be\lbl{and-continue}
      \begin{split}
\frac{1}{2\pi \iu}\oint_c \cR_z [\phi](\xi,\eta,x) dz
   &=
 \frac{-\cD(0;\xi,\eta)}{\mu_{max}\cD^\prime(0)} \bW\frac{1}{2\pi \iu}
    \oint_C \left(\alpha-\bW\right)^{-1} \left[ \Phi \right](0,0,x) d\alpha \\
 &= \frac{-\cD(0;\xi,\eta)}{\mu_{max}\cD^\prime(0)} \bW \bP\left[\Phi \right](0,0,x) \\
 &=   \rho_1\cD(0;\xi,\eta) \bP\left[\Phi \right](0,0,x),
   \end{split}
  \ee

    where $\rho_1=-\left(\cD^\prime(0)\right)^{-1}>0$ by Lemma~\ref{lem.D}. 
      \end{proof}

%%%%%%%%%%%%%%%%%%%%%%%%%%%%%%%%%%%%%%%%%%%%%%%%%%%%%%%%%%%%%%%%%%%%%

\subsection{Asymptotic stability} \label{sec.stability}
We conclude this section with an application to linear asymptotic stability of mixing. For $K\in [0,K_c),$
$\bS$ has no eigenvalues, nonetheless mixing is asymptotically stable in the appropriate topology.

\begin{theorem}\label{thm.asymp}
  Suppose Assumption~\ref{as.decay} holds. Then for $K\in [0, K_c)$ the mixing state is
  linearly asymptotically stable under the
  perturbations from
  $\cH^+_a\subset\cH^-_a$ in the topology of $\cH^-_a$.
\end{theorem}
\begin{proof}
We estimate the semigroups of the linearized systems (\ref{linear}), (\ref{linear-2}).
Each of the operators $\bL_j,\, j\ge 1,$  
generates a continuous semigroup on $\cH_a^+\subset\cH_a^-$ in
the topology of $\cH^-_a$:
\begin{alignat}{1}\lbl{eLj}
  e^{\bL_j t} [\phi ](\xi, \eta, x) &= \phi (\xi + t, \eta + jt, x),\quad j\ge 1,\\
                                        \nonumber
    \lim_{t\to\infty} \|e^{\bL_j t} [\phi] (t)\|_{\cH_a^-}&=0.                                    
\end{alignat}
For $\bL_0$, the same statement holds because of the constraint $w_0(t, \infty, \eta ,x) = 0$.

Since $\bB$ is a bounded operator, $\bS$ generates a continuous semigroup expressed by the Laplace inversion formula for $\phi \in \cH_a^+$:
\be
\begin{split}
 e^{\bS t} [\phi ](\xi, \eta, x)
 &= \frac{1}{2\pi \iu}\lim_{c\to \infty} \int^{b +\iu c}_{b -\iu c}\! 
e^{\lambda t} \cR_\lambda [\phi ] (\xi, \eta, x) d\lambda 
\end{split}
\ee
for some positive $b$.
Let $K\in [0,K_c)$ be arbitrary but fixed.
By Theorem~\ref{thm.crossing}, there exists $\epsilon>0$
such that there are no generalized eigenvalues in $\{-\epsilon \le \Re z \}$
and the generalized resolovent $\cR_\lambda$ is holomorphic in this region.
Hence, the integral pass can be moved to the left half plane and we obtain
\be\lbl{eSt}
\begin{split}
 e^{\bS t} [\phi ](\xi, \eta, x)
 &= \frac{1}{2\pi \iu}\lim_{c\to \infty} \int^{-\epsilon +\iu c}_{-\epsilon -\iu c}\! 
e^{\lambda t} \cR_\lambda [\phi ] (\xi, \eta, x) d\lambda \\
&= \frac{e^{-\varepsilon t}}{2\pi \iu}\lim_{c\to \infty} \int^{c}_{-c}\! \iu e^{\iu \lambda t}
\cR_{i\lambda-\epsilon} [\phi ] (\xi, \eta, x) d\lambda \rightarrow 0 
\end{split}
\ee
in $\cH^-_a$ as $t\to\infty$. 
\end{proof}

      \section{The pitchfork bifurcation}
      \lbl{sec.bifurcation}
      \setcounter{equation}{0}

      In this section, we study a pitchfork bifurcation of mixing underlying the onset of synchronization
      in the second order KM. This bifurcation is identified as the point in the parameter space at which
      a generalized eigenvalue crosses the imaginary axis. Its analysis requires generalized spectral
      theory \cite{Chi15b}.

      \subsection{Assumptions and the main result}\label{sec.assume}
      Throughout this section, we assume the following
      \begin{description}
      \item[(A-1)]
        $\mu_{max}>0$ is a simple eigenvalue of $\bW$ with the corresponding eigenfunction
        $V_{max}$;
      \item[(A-2)] $g$ is an even unimodal probability density function subject to
        Assumption~\ref{as.decay}.
      \end{description}
      Assumption (\textbf{A-1}) holds for many common graph sequences encountered in applications
      including all--to--all, Erd\H{o}s--R{\' e}nyi, small--world graphs, to name a few \cite{ChiMed19a}.
      Assumption (\textbf{A-2}) covers the most representative case. It can be relaxed.
      We impose it for simplicity.

      \begin{remark}\label{rem.negative}
        Suppose $\bW$ has negative eigenvalues. Denote the smallest negative eigenvalue $\mu_{min}$
        and assume that it is simple with the corresponding eigenfunction $V_{min}$. Then after
        setting $W:=-W$ and $K:=-K$, we find that the resultant system satisfies (\textbf{A}) and the
        analysis below shows that the original system undergoes a pitchfork bifurcation
        at $K_c^-=(g_0 \mu_{min})^{-1}<0$.
        \end{remark}
     
      Under these assumptions, as follows from the analysis in Section~\ref{sec.spectral},
      at $K=K_c$, $\bS$ has a simple generalized eigenvalue $\lambda=0$. The corresponding
      eigenfunction is (cf.~\eqref{S-v})
      \begin{equation}\label{vc}
v_0=K_c \lim_{\lambda\to 0+} D(\lambda(\mu_{max}); \xi, \eta) V_{max}\in \cH^-_a.
\end{equation}
      
Finally, we postulate the existence of a one--dimensional center manifold for the vector field
in \eqref{1-8} for $K=K_c$. As we already stated in the Introduction, the proof of this fact is a
very technical problem, which is not
addressed in this paper. The proof of existence of a center manifold in the classical KM
can be found in \cite{Chi15a}.

In the remainder of this section, we perform a center manifold
reduction and show that at $K=K_c$, the system undergoes a pitchfork bifurcation.
The branch of stable steady state solutions bifurcating from mixing
is given in terms of the corresponding values of the order parameter:
\be\lbl{h-asympt}
|h_\infty (x)| =K_c^{-2} (\mu_{max} \rho_2 \rho_3(x))^{-1/2}  \sqrt{K-K_c} +
O(K-K_c),
      \ee
where $\rho_2>0$ and $\rho_3(x)$ are defined below.
In many applications, $\rho_3(x)\equiv 1$ (see Remark~\ref{remark_rho3} for more details).

\subsection{The center manifold reduction}\label{sec.center}
The existence of the one--dimensional center manifold at $K=K_c$ implies the
following Ansatz for the solutions of \eqref{1-8} with $K=K_c+\epsilon^2$,
$0<\epsilon\ll 1$:
      \begin{eqnarray}
       \lbl{An-1}
        w_1 &=&\cP w_1+(\bI-\cP) w_1 = \epsilon c(t) v_0+O(\epsilon^2),\\
        \lbl{An-2}
        w_k&=&q_k(w_1), \; k\in\N_1:=\{ 0 \}\bigcup \{2,3,\dots\},
      \end{eqnarray}
       where $\cP$ is the Riesz projector (cf.~\eqref{RieszS}), $c(t)$ is the function determining the
      location on the slow manifold, and 
     $q_k, k\in\N_1,$ are smooth functions such that $q_k(0)=q^\prime_k(0)=0$.

     Using the Ansatz \eqref{An-1} and \eqref{An-2}, below we derive
      \eqref{h-asympt}. Let $K=K_c+\epsilon^2$ and  rewrite equations for $w_0$, $w_1$, and
      $w_2$ in \eqref{1-8} as follows
      \begin{eqnarray}\lbl{dw0}
        \partial_t w_0 &=&  \partial_{\xi_0} w_0- Ke^{-\xi_0}
                           \left( h(t,x) w_{-1}-\overline{h(t,x)}w_1\right),\\
        \lbl{dw1}
        \partial_t w_1 &=& \bS_0w_1 +\epsilon^2 \left(1-e^{-\xi_1}\right) h(t,x)\hat g(\eta)\\
        \nonumber
        &+ & K \left(1-e^{-\xi_1}\right) \left( h(t,x) w_0-\overline{h(t,x)}w_2\right),\\
        \lbl{dw2}
        \partial_t w_2 &=& \partial_{\xi_2} w_2 + 2\partial_{\eta} w_2
+K(2-e^{-\xi_2}) \left(h(t,x) w_1 - \overline{h(t,x)}w_3\right), 
      \end{eqnarray}
     where  $\xi_i, \; i\in\N,$ are as defined in \eqref{1-6}, and $\bS_0$ is defined by setting $K=K_c$ in
      the expression for $\bS$ (cf.~\eqref{linear}). In particular,
      \begin{equation}\label{def-S0}
      \bS w_1= \bS_0 w_1 +\epsilon^2 \left( 1-e^{-\xi_1}\right) h(t,x) \hat g(\eta).
      \end{equation}

      \begin{lemma}\lbl{lem.w0w2}
        Let
        \begin{eqnarray}\lbl{P0}
P_0(\lambda ; \xi_1, \omega ) &=& \frac{e^{-\xi_1}-1}{\lambda -\iu\omega }
 -\frac{e^{-\xi_1}-1}{\lambda +\iu\omega } - \frac{e^{-\xi_1}-1}{\lambda +1-\iu\omega } 
- \frac{1}{2}\frac{(e^{-\xi_1}-1)^2}{\lambda +1-\iu\omega } \\
\nonumber
                              & +& \frac{e^{-\xi_1}-1}{\lambda +1+\iu\omega }-
                                   \frac{1}{2}\frac{(e^{-\xi_1}-1)^2}{\lambda +1+\iu\omega }, \\
  \lbl{P2}
P_2(\lambda ; \xi_1, \omega ) &=& \frac{1}{(\lambda -\iu \omega )^2} 
   + \frac{1}{2}\frac{(e^{-\xi_1}+1)^2}{(\lambda +1 -\iu \omega )^2}+\frac{1}{\lambda -\iu\omega }
                                  -\frac{1}{\lambda +1-\iu \omega } \\
  \nonumber 
& -& \frac{e^{-\xi_1}+1}{\lambda -\iu\omega }
     - \frac{2(e^{-\xi_1}+1)}{\lambda +1/2 - \iu\omega } +
     \frac{3(e^{-\xi_1}+1)}{\lambda +1 - \iu\omega },\\
          \lbl{P1}
          P_1(\lambda ; \xi_1, \eta ) &=& \int_\R \left(\lambda-\bL_1\right)^{-1} \left[
                                            \left(1-e^{-\xi_1}\right) \left(P_0-P_2\right) e^{\iu \eta\omega}
                                            \right] g(\omega) d\omega.
\end{eqnarray}
        On the center manifold, we have
        \begin{eqnarray}
          \lbl{htx}
          h(t,x) &=& \epsilon c(t) V_{max}(x) + O(\epsilon ^2),\\
          \lbl{w0}
w_0 &=& K_c^2\epsilon ^2 |c(t)|^2 |V_{max}|^2 
        \lim_{\lambda \to 0+} \int_{\R} P_0(\lambda ; \xi_1, \omega )
        e^{\iu\eta \omega }g(\omega )d\omega
        + O(\epsilon^3), \\
          \lbl{w2}
 w_2 &=&  K_c^2 \epsilon^2 c(t)^2 V_{max}^2 
         \lim_{\lambda \to 0+} \int_{\R} P_2(\lambda ; \xi_1, \omega ) e^{\iu\eta \omega }
         g(\omega )d\omega + O(\epsilon^3).
\end{eqnarray}
In addition,
\be\lbl{rho2}
\rho_2=-\lim_{\lambda\to 0+} P_1(\lambda;0,0) >0.
\ee
\end{lemma}
The number $\rho_2$ is determined solely by $g(\omega )$. Its expression will be given in the end of Sec.\ref{sec.proof}.
         With Lemma~\ref{lem.w0w2} in hand, we now turn to the center manifold reduction.
          The proof of the lemma is given in the next subsection.

          We want to project both sides of \eqref{dw1} onto the eigenspace spanned by $v_0$. To this end,
          we note
          \be\lbl{note-1}
          \cP [\p_t w_1] = \epsilon \dot c(t) v_0,\quad\mbox{and}\quad \cP[\bS_0 w_1]=0,
          \ee
           as follows from \eqref{An-1}.
          Further, from Lemma~\ref{lem.R},
          \be\lbl{note-2}
          \cP\left[\left(1-e^{-\xi_1}\right) h(t,x)\hat g(\eta)\right]=
          \rho_1 \lim_{\lambda\to 0+} \left\{ \cD(\lambda;\xi_1,\eta)  \left. \bP \left[
            \left(\lambda-\bL_1\right)^{-1} \left( 1-e^{-\xi_1}\right) \hat g(\eta) h(t,x) \right]\
            \right|_{\xi_1= \eta=0}\right\}.
          \ee
              By the first line in \eqref{compute-D},
          \be\lbl{by-first}
          \left(\lambda-\bL_1\right)^{-1} \left( 1-e^{-\xi_1}\right) \hat g(\eta)=\cD(\lambda; \xi_1,\eta).
          \ee
          After plugging \eqref{by-first} and \eqref{htx} into \eqref{note-2},
          we have
          \be\lbl{first-group}
          \begin{split}
            \cP \left[ \left( 1-e^{-\xi_1} \right) h(t,x)\hat g(\eta) \right] &=
            \rho_1\lim_{\lambda\to 0+} \left\{ \cD(\lambda; \xi_1, \eta)
              \bP \left[ \cD(\lambda; \xi_1, \eta) h(t,x) \right]_{\xi_1=\eta=0}\right\} \\
        &=
            \rho_1\lim_{\lambda\to 0+} \left\{ \cD(\lambda; \xi_1, \eta) \cD(\lambda)\right\}
              \bP\left[ \epsilon c(t) V_{max}\right] (0,0,x)+O(\epsilon^2)\\
            &= \epsilon \rho_1 c(t) V_{max}(x) \lim_{\lambda\to 0+}
            \left\{ \cD(\lambda; \xi_1, \eta) \cD(\lambda) \right\} +O(\epsilon^2)\\
          &=\frac{\epsilon\rho_1 c(t) v_0}{K_c^2 \mu_{max}} +O(\epsilon^2),
        \end{split}
        \ee
        where we used \eqref{change-var2} and \eqref{vc} to obtain the last equality.
        
Finally, we turn to the last group of terms of the right--hand side of \eqref{dw1}.
By Lemma~\ref{lem.w0w2},
          \be\lbl{substract}
          \begin{split}
          h(t,x) w_0-\overline{h(t,x)} w_2& =
          \epsilon^3 K_c^2 c(t)|c(t)|^2 V_{max} |V_{max}|^2 \times\\
          & \times \lim_{\lambda\to 0+}
          \int_\R \left[P_0-P_2\right](\lambda;\xi_1,\omega)
          e^{\iu \eta\omega} g(\omega)d\omega + O(\epsilon^4).
\end{split}
          \ee
         Further,
          \be\lbl{continue}
          \begin{split}
            \left(\lambda-\bL_1\right)^{-1}\left[(1-e^{-\xi_1})
              \left(h w_0-\overline{h} w_2\right) \right]&=
            \epsilon^3 K_c^2 c(t)|c(t)|^2 V_{max} |V_{max}|^2 
          \lim_{\lambda\to 0+}
              P_1(\lambda;\xi_1,\eta)\\
             & +O(\epsilon^4),
            \end{split}
            \ee
            where $P_1$ is defined in \eqref{P1}.
This yields
\be\lbl{this-yields}
\begin{split}
  \cP \left[(1-e^{-\xi_1})
    \left(h w_0-\overline{h} w_2\right) \right] &= \rho_1\lim_{\lambda\to 0+}
\left\{ \cD(\lambda; \xi_1, \eta) \bP \left[ \left(\lambda-\bL_1\right)^{-1}
    \left(1-e^{-\xi_1}\right)\left( h w_0-\bar{h} w_2\right)\right]_{\xi_1=\eta=0}\right\}\\
&=\epsilon^3 \rho_1 K_c^2 c(t) |c(t)|^2 \lim_{\lambda\to 0+}
\left\{ \cD(\lambda; \xi_1, \eta) \bP \left[ V_{max} |V_{max}|^2\right] P_1(\lambda; 0,0)
  \right\}
  +O(\epsilon^4)\\
  &=
\epsilon^3 \rho_1 K_c c(t)|c(t)|^2
  v_0 \frac{\bP \left[ V_{max}|V_{max}|^2\right]}
    {V_{max}} \lim_{\lambda\to 0+} P_1(\lambda; 0,0) + O(\epsilon^4).
  \end{split}
  \ee

  By projecting both sides of \eqref{dw1} onto the subspace spanned by $v_0$ and
  using \eqref{note-1}, \eqref{first-group}, and \eqref{this-yields}, we obtain
\begin{eqnarray*}
\epsilon \dot{c} v_0 = \frac{\epsilon^3 \rho_1}{K_c^2 \mu_{max}} c(t) v_0
  +\epsilon^3 K_c^2 \rho_1 c(t) |c(t)|^2 v_0 \frac{\bP [ V_{max} |V_{max}|^2]}{V_{max}}
     \lim_{\lambda\to 0+} P_1(\lambda; 0,0) + O(\epsilon^4).
\end{eqnarray*}
This yields
  \be\label{center}
  \dot c=\frac{\epsilon^2\rho_1}{K_c^2 \mu_{max}} c(t)
  \left( 1+ K_c^4 \mu_{max} \lim_{\lambda\to 0+} P_1(\lambda; 0,0)
    \frac{\bP [ V_{max} |V_{max}|^2]}{V_{max}}  |c(t)|^2 \right) + O(\epsilon^3).
    \ee
Equation \eqref{center} describes  the dynamics on the center manifold.
Recalling that $ h(t,x) =\epsilon c(t) V_{max}+O(\epsilon^2)$ from Lemma~\ref{lem.w0w2}, 
we rewrite this equation to the equation 
of the local order parameter $h = h(t,x)$ as
\begin{eqnarray}
\frac{dh}{dt} &=& \frac{\rho_1}{K_c^2 \mu_{max}} h 
  \left( \epsilon^2 + K_c^4 \mu_{max} \lim_{\lambda\to 0+} P_1(\lambda; 0,0) 
   \frac{\bP [ V_{max} |V_{max}|^2]}{V_{max}|V_{max}|^2} |h|^2  \right) + O(\epsilon ^4) \nonumber \\
&=& \frac{\rho_1}{K_c^2 \mu_{max}} h 
  \left( \epsilon^2 - K_c^4 \mu_{max} 
     \rho_2 \rho_3  |h|^2  \right) + O(\epsilon ^4)\lbl{h-ode}
\end{eqnarray}
    where $\rho_2>0$ is defined in \eqref{rho2} and $\rho_3$ is defined by
    \be\lbl{rho-3}
    \rho_3(x) := \frac{\bP [ V_{max}(x) |V_{max}(x)|^2]}{V_{max}(x)|V_{max}(x)|^2}.
    %=\frac{\int_0^1 V_{max}(s) |V_{max}(s)|^2 ds}
    %{\int_0^1V_{max}(s)^2 ds} \equiv 1.
\ee
Equation \eqref{h-ode} has a fixed point given by \eqref{h-asympt}. It undergoes a  pitchfork bifurcation at $K=K_c$.
Since $\rho_1$ and $\rho_2$ are positive, the fixed point is stable when $\rho_3>0$ (see below).

\begin{remark}\lbl{remark_rho3}
In many applications, a graphon is of the form $W(x,y) = G(x-y)$ such that $G(x) = G(-x)$.
For instance, the family of small-world graphs is defined by the graphon of this type
\cite{Med14c, ChiMed19a}. In this case,
$W$ admits Fourier series expansion
\begin{eqnarray}
W(x,y) = \sum_{k\in \mathbb{Z}} c_k e^{2\pi \iu k (x-y)},\, c_{k} = c_{-k}\in \mathbb{R},
\quad \sum_{k\in \mathbb{Z}}c_k^2 < \infty.
\end{eqnarray}
Then, eigenvalues of the operator $\bW$ are the coefficients of the expansion. 
The largest eigenvalue is given by
\begin{equation}
\mu_{max} = c_m := \sup \{ c_k\, : \, k\in \mathbb{Z} \},
\end{equation}
and the corresponding eigenfunction is $V_{max} = e^{2\pi \iu m x}$~\cite{ChiMed19b}.
In this case, $\rho_3 = 1$.
For many important graphs such as Erd{\H o}s--R{\' e}nyi and small-world graphs, $\mu_{max} = c_0$
and $V_{max} = \rho_3 = 1$. In this case, the stable branch of equilibria is given through the
values of the order parameter
\be\lbl{h-asympt2}
 |h_\infty (x)| =K_c^{-2} (\mu_{max} \rho_2)^{-1/2}  \sqrt{K-K_c} +
O(K-K_c).
      \ee
\end{remark}

\subsection{Proof of Lemma~\ref{lem.w0w2}}\lbl{sec.proof} 
   By plugging \eqref{An-1} and \eqref{An-2} into \eqref{dw1} and recalling $\bS_0v_0=0$, we
      immediately get
      \be\lbl{dw1-epsilon}
      \partial_t w_1=O(\epsilon^2).
      \ee
      From this, \eqref{An-1} and \eqref{An-2}, we further observe that
      \be\lbl{dwk-epsilon}
      \partial_t w_k = q_k'(w_1) \partial_t w_1 =O(\epsilon^3),\quad k\in\N_1.
      \ee
      From \eqref{An-1} and the equation for   $h$ in \eqref{1-8}, we have 
      \begin{equation*}
        \begin{split}
h(t,x) &= \int_{I}\! W(x,y) (\epsilon c(t) v_0(0,0,y) + O(\epsilon^2)) dy \\
&= \epsilon c(t) K_c \lim_{\lambda \to 0+} D(\lambda ) \int_{I}\! W(x,y)  V_{max}(y) dy + O(\epsilon^2) \\
&= \epsilon c(t) K_c \mu_{max}  \lim_{\lambda \to 0+} D(\lambda ) V_{max}(x) + O(\epsilon^2) \\
&= \epsilon c(t) V_{max}(x)+O(\epsilon^2).
\end{split}
\end{equation*}
This shows \eqref{htx}.

 Next we turn to the equation for $w_2$:
\begin{equation*}
\partial_t w_2 = \partial_{\xi_2} w_2 + 2\partial_\eta w_2
+K(2-e^{-\xi_2}) (h w_1 - \overline{h}w_3).
\end{equation*}
Since $\partial_t w_2$ and $\overline{h}w_3$ are of order $O(\epsilon^3)$, we have
\begin{equation}
\left( \partial_{\xi_2} + 2\partial_\eta\right) w_2 
= -K_c \epsilon^2 c(t)^2 (2-e^{-\xi_2}) V_{max} v_0 + O(\epsilon^3).
\label{w2-eqn}
\end{equation}

\begin{lemma}
 On the center manifold, $w_2$ is expressed as
\begin{eqnarray}
w_2 &=& K_c^2 \epsilon^2 c(t)^2 V_{max}^2  \nonumber \\ 
&\times &\lim_{\lambda \to 0+} 
    \int_{\mathbb{R}}\! \Bigl( \frac{1}{(\lambda -\iu \omega )^2} 
   + \frac{1}{2}\frac{e^{-2\xi_2}}{(\lambda +1 -\iu \omega )^2}+\frac{1}{\lambda -\iu \omega }-\frac{1}{\lambda +1-\iu \omega } \nonumber \\ 
&-& \frac{e^{-\xi_2}}{\lambda -\iu \omega }
 - \frac{2e^{-\xi_2}}{\lambda +1/2 - \iu \omega } + \frac{3e^{-\xi_2}}{\lambda +1 - \iu\omega } \Bigr)
 e^{i\eta \omega }g(\omega )d\omega + O(\epsilon^3) .\qquad
\label{lemma4-3}
\end{eqnarray}
\end{lemma}

\begin{proof}
 A straightforward substitution shows that (\ref{lemma4-3}) satisfies (\ref{w2-eqn}) up to $O(\epsilon )^3$.
\end{proof}
By the relation $e^{-\xi_2} = e^{-\xi_1} + 1$ derived from \eqref{1-6}, we obtain \eqref{w2} with \eqref{P2}.

Next, the equation for $w_0$ is given by
\begin{equation*}
\partial_t w_0 = \partial_{\xi_0} w_0 -Ke^{-\xi_0}\left(h w_{-1} - \overline{h}w_1\right).
\end{equation*}

\begin{lemma}
On the center manifold, $w_0$ is expressed as
\begin{eqnarray}
& & w_0 = K_c^2\epsilon^2 |c(t)|^2 \cdot |V_{max}|^2 \nonumber \\
&\times & \lim_{\lambda \to 0+} \int_{\mathbb{R}}\! \Bigl( \frac{e^{-\xi_0}}{\lambda -\iu \omega }
 -\frac{e^{-\xi_0}}{\lambda + \iu \omega } - \frac{e^{-\xi_0}+e^{-2\xi_0}/2}{\lambda +1-\iu\omega } 
 + \frac{e^{-\xi_0} - e^{-2\xi_0}/2}{\lambda +1+\iu \omega } \Bigr) e^{\iu \eta \omega } g(\omega )d\omega \nonumber \\
& &   + O(\epsilon^3).
\end{eqnarray}
\end{lemma}

\begin{proof}
By using $w_{-1}(t, \xi_{-1}, \eta , x) = \overline{w_1(t, \xi_{-1}, -\eta, x)}$
and $\partial w_0/\partial t = O(\epsilon^3)$, we get
\begin{eqnarray*}
\frac{\partial w_0}{\partial \xi_0} &=& K_c e^{-\xi_0} \cdot \epsilon^2 |c(t)|^2 
   \left(  V_{max} \overline{v_0(\xi_{-1},-\eta, x)} - \overline{V_{max}} v_0(\xi_1, \eta, x) \right) + O(\epsilon^3) \\
&=& K_c^2\epsilon^2 |c(t)|^2\cdot |V_{max}|^2 e^{-\xi_0} \\
&\times & \lim_{\lambda \to 0+} 
\int_{\mathbb{R}}\! \Bigl( \frac{1}{\lambda +\iu \omega }
 -\frac{1}{\lambda -\iu \omega } - \frac{1-e^{-\xi_0}}{\lambda +1+\iu \omega } 
 + \frac{1 + e^{-\xi_0}}{\lambda +1-\iu \omega } \Bigr) e^{\iu \eta \omega } g(\omega )d\omega + O(\epsilon^3),
\end{eqnarray*}
where the relations $e^{-\xi_{\pm 1}} = 1\pm e^{-\xi_0}$ from \eqref{1-6} is used to write the right hand side as a function of $\xi_0$.
Integrating both sides in $\xi_0$ with the boundary condition $w_0(t, \infty, \eta, x) = 0$ verifies the lemma.
\end{proof}

The relation $e^{-\xi_{1}} = 1+ e^{-\xi_0}$ gives \eqref{w0} with \eqref{P0}.

Finally, we prove \eqref{rho2}. Recall
\begin{eqnarray*}
P_1(\lambda ; \xi_1, \eta ) &:=& \int_\R \left(\lambda-\bL_1\right)^{-1} \left[
                                            \left(1-e^{-\xi_1}\right) \left(P_0-P_2\right) e^{\iu \eta\omega}
                                            \right] g(\omega) d\omega
\end{eqnarray*}
By (\ref{L-res}), 
\begin{eqnarray*}
& & \left(\lambda-\bL_1\right)^{-1}
   \left[ \left(1-e^{-\xi_1}\right) \left(P_0-P_2\right) e^{\iu \eta\omega}\right] (\xi_1, \eta, x) \\
&=& \int_{\R^2} \frac{e^{\iu (\xi_1 \zeta + \eta \widetilde{\omega })}}{\lambda -\iu \zeta - \iu \widetilde{\omega }}\,
  \mathcal{F}[(1-e^{-\xi_1}) \left(P_0-P_2 \right) e^{\iu \eta \omega}](\zeta, \widetilde{\omega})d\zeta d\widetilde{\omega }.
\end{eqnarray*}
Here, $\mathcal{F}[(1-e^{-\xi_1}) \left(P_0-P_2 \right) e^{\iu \eta \omega}]$ is the Fourier transform 
with respect to $\xi_1$ and $\eta$ with parameters $\lambda , \omega $.
Thus, 
\begin{eqnarray*}
& & \left(\lambda-\bL_1\right)^{-1}
   \left[ \left(1-e^{-\xi_1}\right) \left(P_0-P_2\right) e^{\iu \eta\omega}\right] (\xi_1, \eta, x) \\
&=& \int_{\R^2} \frac{e^{\iu (\xi_1 \zeta + \eta \widetilde{\omega })}}{\lambda -\iu\zeta - \iu \widetilde{\omega }}\,
  \mathcal{F}[(1-e^{-\xi_1}) \left(P_0-P_2 \right) ](\zeta)\cdot \delta (\widetilde{\omega } - \omega ) d\zeta d\widetilde{\omega } \\
&=& \int_{\R} \frac{e^{\iu (\xi_1 \zeta + \eta \widetilde{\omega })}}{\lambda -\iu \zeta - \iu\omega }\,
  \mathcal{F}[(1-e^{-\xi_1}) \left(P_0-P_2 \right) ](\zeta) d\zeta.
\end{eqnarray*}
Hence, we obtain
\begin{eqnarray*}
 P_1(\lambda ; \xi_1, \eta ) &=& \int_{\R^2} \frac{e^{\iu (\xi_1 \zeta + \eta \omega })}{\lambda -\iu \zeta - \iu \omega }\,
  \mathcal{F}[(1-e^{-\xi_1}) \left(P_0-P_2 \right) ](\zeta ) g(\omega )d\zeta d\omega\\
P_1(\lambda ; 0, 0) &=& \int_{\R^2} \frac{1}{\lambda - \iu \zeta - \iu \omega }\,
  \mathcal{F}[(1-e^{-\xi_1}) \left(P_0-P_2 \right) ](\zeta )g(\omega )d\zeta d\omega,
\end{eqnarray*}
where $\mathcal{F}[(1-e^{-\xi_1}) \left(P_0-P_2 \right) ]$ is the Fourier transform with respect to $\xi_1$.
Here $(1-e^{-\xi_1}) \left(P_0-P_2 \right)$ is a linear combination of $1, e^{-\xi_1}, e^{-2\xi_1}, e^{-3\xi_1}$.
Hence, $\mathcal{F}[(1-e^{-\xi_1}) \left(P_0-P_2 \right) ]$ is a linear combination of
$\delta (\zeta), \delta (\zeta - \iu ), \delta (\zeta - 2\iu ), \delta (\zeta - 3\iu )$.
By integrating with respect to $\zeta$, after  
a straightforward albeit lengthy calculation we obtain
\begin{equation*}
  \begin{split}
P_1(\lambda ;0,0) & = \int_{\R}\! \left( \frac{1}{(\lambda -\iu\omega )(\lambda +\iu\omega )}
   -\frac{1}{(\lambda -\iu\omega )^3} \right) g(\omega )d\omega \\
 &+    \frac{3}{2}\int_{\R}\! \left(  \frac{1}{\lambda -\iu \omega }-\frac{1}{\lambda +\iu \omega }
 \right) g(\omega )d\omega \\
&- \int_{\R}\! \left(
  \frac{12 \iu \omega ^7 + 4\omega ^6+43 \iu\omega ^5+13\omega ^4+58i\omega^3+25\omega ^2+
    21\iu\omega +10}
  {(1+\omega ^2)^3(1+4\omega ^2)} \right) g(\omega )d\omega.
\end{split}
\end{equation*}
Since $g$ is an even function,
\begin{equation*}
  \begin{split}
 P_1(\lambda ;0,0) 
&= \int_{\R}\! \frac{1}{\lambda -\iu \omega } \frac{g(\omega )}{\iu \omega }d\omega
  + \frac{1}{2}\int_{\mathbb{R}}\! \frac{1}{\lambda -\iu \omega } g''(\omega )d\omega  \\
& - \int_{\R}\! \left(
 \frac{4\omega ^6 + 13\omega ^4 + 25\omega ^2 +10}
   {(1+\omega ^2)^3(1+4\omega ^2)} \right) g(\omega )d\omega .
\end{split}
\end{equation*}
Further, Sokhotski--Plemelj formula \cite{Simon-Complex} shows
\begin{equation*}
  \begin{split}
\lim_{\lambda \to 0+} P_1(\lambda ;0,0) 
&= \frac{1}{2}\pi g''(0) + 2 \lim_{\varepsilon \to 0} \int^{\infty}_{\varepsilon }\! \frac{g'(\omega )}{w}d\omega   \\
&  - \int_{\R}\! \left(
 \frac{4\omega ^6 + 13\omega ^4 + 25\omega ^2 +10}
   {(1+\omega ^2)^3(1+4\omega ^2)} \right) g(\omega )d\omega,
\end{split}
\end{equation*}
which is a negative number and $\rho_2$ is positive.

\begin{figure}
  \centering
  \textbf{a} \includegraphics[width=.47\textwidth]{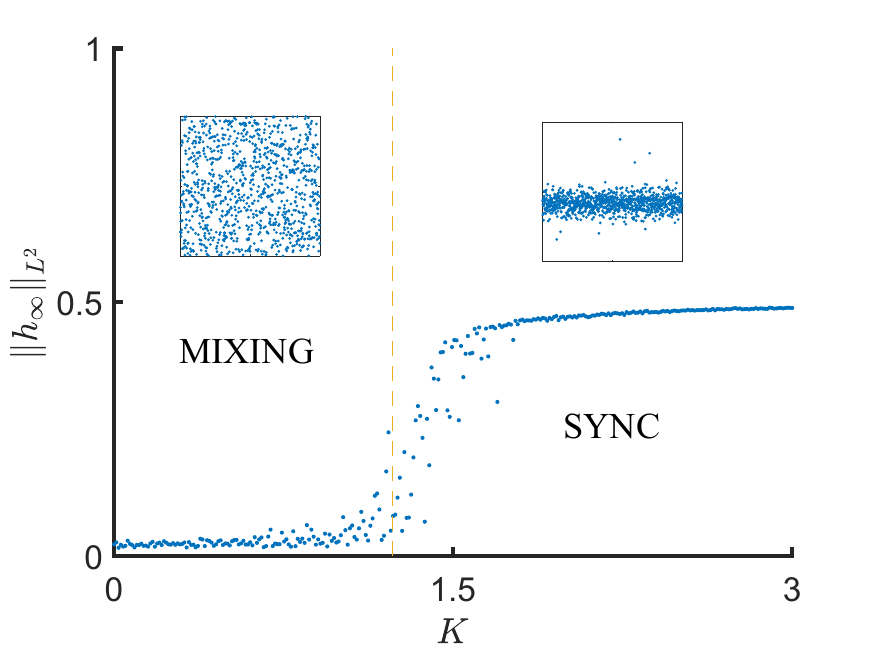}
  \textbf{b} \includegraphics[width=.47\textwidth]{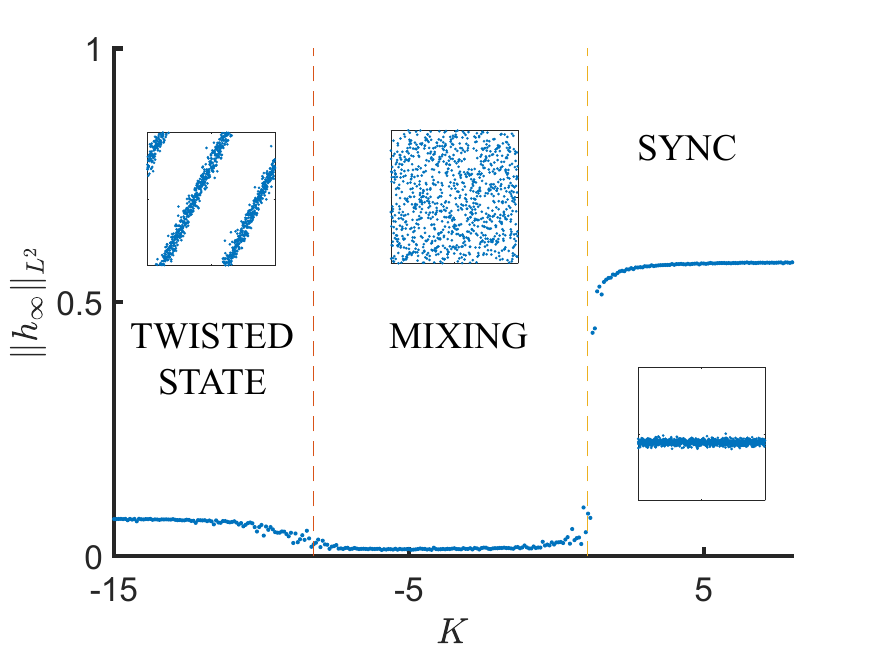}
  \caption{ The bifurcation diagrams for the second order KM on ER  (\textbf{a}) and SW
    (\textbf{b}) random graphs. In these experiments, the intrinsic
    frequences are taken from the normal distributrion with mean $0$ and standard deviation equal
    to $0.3$.
    The parameters used for generating ER and SW graphs are
    $p=0.5$ for the former and $p=0.1$ and $r=0.3$ for the latter.
    The plot in (\textbf{a}) shows a pitchfork bifurcation at $K_c>0$. The bifurcation diagram
    for the small-world family shows two bifurcations at $K_c^+>0$ and $K_c^-<0$. The latter features a spatially
    heterogeneous pattern bifurcating from mixing, because the eigenfunction of $\bW$ corresponding to the smallest
    negative eigenvalue $\mu_{min}$ is not constant (see text for details). The red dashed lines indicate
    pitchfork bifurcations identified analytically. These values are in good agreement with the results of numerical simulations.
   } 
             	\label{f.2}
             \end{figure}

%%%%%%%%%%%%%%%%%%%%%%%%%%%%%%%%%%%%%%%%%%%%%%%%%%%%%%%%%%%

\section{Examples}\lbl{sec.numerics}
\setcounter{equation}{0}
In this section, we apply the theory developed in the previous sections to identify the pitchfork bifurcations
in the KM with inertia on Erd{\H o}s-R{\' e}nyi (ER) and small-world (SW) graphs.  The
eigenvalues and eigenfunctions of the integral operator $\bW$ corresponding to these graphs 
were computed in our previous work \cite{ChiMed19a}. Below, we will use this information without any further
explanations. An interested reader is referred to \cite{ChiMed19a} for more details.

% By varying the shape of the intrinsic frequency distribution and network connectivity, the bifurcation
% of mixing in the KM of inertia can lead to a variety of spatial patterns from synchronization to
% twisted states and clusters, to chimera states \cite{CMM20a}. The bifurcation analysis of this paper
% shows a way for analytical understanding of these patterns.

\subsection{ER graphs}
An ER graph  $G(n,p)$ is a graph on $n$ nodes, for which the probability of two given nodes being connected by
an edge is equal to $p\in (0,1)$. The edges between distinct pairs of nodes are assigned independently. $G(n,p)$ is a convergent
sequence of graphs, whose limit is given by a constant function $W\equiv p$.
The largest eigenvalue of $\bW$ is $\mu_{max}=p$, and the corresponding eigenfunction is constant $V_{max}\equiv 1$
(cf. \cite{ChiMed19a}). Thus,  there is a pitchfork bifurcation at 
             $$
             K_c=\frac{1}{pg_0},
             $$
             where $g_0$ is defined in \eqref{g0}.
The bifurcation diagram in Fig.~\ref{f.2}~\textbf{b} shows that in the  vicinity of $K_c$
             the order parameter undergoes a sharp transition from very small positive values to the values close
             to $0.5$. This corresponds to the loss of stability of mixing and the onset of synchronization.
             Since $V_{max}$ is constant, the unstable mode is spatially homogeneous. 
             With minor modifications the analysis of the pitchfork bifurcation can be extended to sparse
             Erd{\H o}s--R{\' e}nyi graphs  \cite{Med19}.
             
             \subsection{SW graphs}
             There are two types of random connections in a SW graph: the short-range and the long-range. The former are
             assigned with probability $1-p$ and the latter are assigned with probability $p\in (0,1/2)$.
             The connectivity is defined by another parameter $r\in (0, 1/2)$ (see \cite{Med14c} for the exact definition the
             SW graph sequence used here). SW graphs in \cite{Med14c} are defined as W-random graphs with the limiting graphon
             $$
             W(x,y)=\left\{\begin{array}{cc}
                             1-p,& d(x,y)<r,\\
                             p,& \;\mbox{otherwise}.
                           \end{array}
                         \right.,
                         $$
                         where $d(x,y)=\min\{|x-y|, 1-|x-y|\}$. The eigenvalues and the corresponding eigenfunctions of
                         $\bW$ for the SW graphs can be found  in \cite{ChiMed19a}. In particular, the largest positive
                         eigenvalue
                         $
                         \mu_{max}=2r+p-4rp
                         $
                         is simple and the corresponding eigenfunction $V_{max}\equiv 1$. This yields a pitchfork bifurcation
                         at $K_c^+=\left(g_0 (2r+p-4rp)\right)^{-1}.$

                         In addition, $\bW$ has negative
                         eigenvalues. As was explained in Remark~\ref{rem.negative}, if the smallest negative eigenvalue
                         $\mu_{min}$ is simple, then there is another bifurcation at $K_c^-=\left( g_0 \mu_{min}\right)^{-1}$.
                         The value of $\mu_{min}$ has the following
                         form
                         $$
                         \mu_{min}=(\pi k^\ast)^{-1} (1-2p)\sin\left(2\pi k^\ast\right)
                         $$
                         for some integer $k^\ast\neq 0$, which is not available analytically in general.
                         The corresponding eigenfunction $V_{min}=e^{2\pi\iu k^\ast x}$
                           is not constant anymore. This implies that in contrast to the bifurcation at $K_c^+$, the steady state
                           bifurcating from mixing at $K^-_c<0$ is heterogeneous (Fig.~\ref{f.2}\textbf{b}).
                          In \cite{ChiMed19b}, a bifurcation for a non-constant eigenfunctions and the corresponding twisted state is studied for the classical KM in detail.

\vskip 0.2cm
\noindent
{\bf Acknowledgements.} The authors thank Matthew Mizuhara for
providing numerical bifurcation diagrams. This work was supported in part by NSF grant DMS 2009233 (to GSM).
        
% \bibliographystyle{amsplain}
% \bibliography{mix20}

\def\cprime{$'$} \def\cprime{$'$}
\providecommand{\bysame}{\leavevmode\hbox to3em{\hrulefill}\thinspace}
\providecommand{\MR}{\relax\ifhmode\unskip\space\fi MR }
% \MRhref is called by the amsart/book/proc definition of \MR.
\providecommand{\MRhref}[2]{%
  \href{http://www.ams.org/mathscinet-getitem?mr=#1}{#2}
}
\providecommand{\href}[2]{#2}

\end{document}